\documentclass[oneside,english]{amsart}
\usepackage[T1]{fontenc}
\usepackage[latin9]{inputenc}
\usepackage{xcolor}
\usepackage{mathrsfs}
\usepackage{amstext}
\usepackage{amsthm}
\usepackage[small]{diagrams}

\def\R{\mathbb R}
\def\V{\mathcal{V}}
\def\pa{\partial}
\def\id{\mathrm{id}}
\def\H{\mathcal{H}}
\def\E{\mathcal{E}}

\def\ker{\mathrm{Ker}}
\def\im{\mathrm{Im}}

\theoremstyle{plain}
\newtheorem{theorem}{Theorem}%[section]
\newtheorem{lemma}{Lemma}%[section]
\newtheorem{corollary}{Corollary}

\theoremstyle{remark}
\newtheorem{remark}{Remark}
\newtheorem{example}{Example}%[section]

\usepackage{amsthm}

\makeatletter
\renewenvironment{proof}[1][\proofname]{\par
  \normalfont \topsep6\p@\@plus6\p@\relax
  \trivlist
  \item[\hskip\labelsep
        \itshape
    #1\@addpunct{.}]\ignorespaces
}{%
  \endtrivlist\@endpefalse
}
\makeatother
\begin{document}

\title{Variational submanifolds of Euclidean spaces
}

\author{D. Krupka}
\address{Demeter Krupka\newline Lepage Research Institute, University of Presov\newline 17th November 1, Presov, Slovakia}
\email{demeter.krupka@lepageri.eu}

\author{Z. Urban}
\author{J. Voln\'a}
\address{Zbyn\v{e}k Urban and Jana Voln\'a\newline Department of Mathematics and Descriptive Geometry\newline
        V\v{S}B-Technical University of Ostrava\newline
        17. listopadu 15, 708 33 Ostrava-Poruba, Czech Republic
}
\email{zbynek.urban@vsb.cz; jana.volna@vsb.cz}

\maketitle

\begin{abstract}
Systems of ordinary differential equations (or dynamical forms in Lagrangian mechanics), induced by embeddings of smooth fibered manifolds over one-dimensional basis, are considered in the class of variational equations. For a given non-variational system, conditions assuring variationality (the Helmholtz conditions) of the induced system with respect to a submanifold of a Euclidean space are studied, and the problem of existence of these ``variational submanifolds" is formulated in general and solved for second-order systems. The variational sequence theory on sheaves of differential forms is employed as a main tool for analysis of local and global aspects (variationality and variational triviality). The theory is illustrated by examples of holonomic constraints (submanifolds of a configuration Euclidean space) which are variational submanifolds in geometry and mechanics.
\end{abstract}

\section{Introduction}
\label{intro}
Our main objective in this article is to study restrictions of \textit{simple integral variational functionals}, defined on Euclidean spaces, to submanifolds of these spaces. The submanifolds may have specific topological structure; the question is what can be said in general about the influence of the topology on \textit{variationality} of differential equations, given on submanifolds, and on the existence and structure of the corresponding \textit{local} and \textit{global} variational principles. Specifically, given a system of (ordinary) differential equations on a Euclidean space, \textit{not} necessarily variational, we search for \textit{variational submanifolds} on which the constrained system becomes (locally or globally) variational. This concept within the theory of variational differential equations is new, and belongs to transformations of differential equations into variational ones; see Voicu and Krupka \cite{Voicu}, and references therein.

The method how to understand of what is going on when we restrict
an integral variational functional to a submanifold is the \textit{variational sequence theory}. Recall for example that a basic concept of the calculus of variations, the \textit{Euler-Lagrange mapping}, can be constructed (for variational functionals on fibered spaces) as the quotient mapping of the exterior derivative operator $d$, acting on differential forms, by the restriction of $d$ to the so called \textit{contact forms}. This construction can, in principle, be described in a simple way: the de Rham sequence on the corresponding underlying manifold should be factored through its \textit{contact subsequence}. Then it turns out, in particular, that one of the quotient morphisms coincides with the Euler-Lagrange mapping. The quotient (sheaf) sequence, the \textit{variational sequence}, then can be used to study the local and global properties of the Euler-Lagrange mapping. Note, however, that the variational sequence includes much more relevant local and global information on the variational functionals on fibered spaces.

To study the variationality concepts on submanifolds, we shall focus our attention to the restriction of variational sequences, defined over fibrations $\R\times \R^m$ over the \textit{real line} $\R$, representing \textit{time}, to subfibrations $\R\times Q$, where $Q$ is an embedded submanifold of $\R^m$ (the \textit{constraint submanifold}). This structure of underlying spaces corresponds from variational point of view to \textit{simple integral} problems for curves from $\R$ to $\R^m$ or from $\R$ to $Q$.

Section 2 is devoted to basic facts, related with the notion of a contact differential form on some jet prolongation of a fibered manifold. We extend these concepts to fibered manifolds smoothly embedded in a Euclidean space, and describe properties of the pull-back of contact forms.

Section 3 is divided in three parts. First we give a brief summary of basic assertions of the theory of variational sequences over one-dimensional bases (real line $\R$, circle $S^1$). Next we study a modification of the variational sequence theory to embedded submanifolds in Euclidean spaces. Our main purpose is to characterize global aspects of extremal equations subjected to holonomic constraints. In the third part we recall relations of variational sequences and the calculus of variations.

In Section 4 we investigate systems of second-order ordinary differential equations, which become (locally or globally) variational when contracted to an embedded submanifold of the Euclidean space $\R^m$. The general problem to be solved is twofold: (a) to determine the structure of a system such that the induced system is variational, and (b) to characterize the class of submanifolds, which induce variational systems. Our main results are summarized in two theorems, stating the corresponding Helmholtz-type conditions.

Section 5 contains some examples of holonomic systems and the corresponding \textit{variational submanifolds}. On one side we discuss the structure of constrained systems and obstructions for global variationality for topologically \textit{non}-equivalent, embedded submanifolds of $\R^3$ (the sphere $S^2$ and the M\"obius strip $M_{r,a}$). On the other side, we study the existence of variational submanifolds for concrete examples of second-order mechanical systems (e.g. the dumped oscillator, the gyroscopic-type system) and, in particular, systems possessing the spheres $S^1$ and $S^2$ as the variational submanifolds are obtained.
%%%%%

Our approach to variational submanifolds in its simplest setting
describes basic ideas of the theory of variational submanifolds as a part
of the variational sequence theory. It can be extended in several
directions, for instance to different underlying structures, related with
multiple integrals, partial differential equations, variational principles
in field theory with constraints, and variational sequences on Grassmann
fibrations, where the variationality conditions of the Helmholtz type are
available.

Basic reference for variational sequence theory in fibered manifolds over one-dimensional basis is Krupka \cite{SeqMech} (for Grassmann fibrations see Urban and Krupka \cite{Acta} and Urban \cite{Urban}). For the Helmholtz variationality conditions in mechanics and field theory we refer to Anderson and Duchamp \cite{And}, Krupka \cite{Praha}, and Sarlet \cite{Sarlet}. General variational sequence theory is explained in Krupka \cite{Seq}, \cite{Book}. In all these sources also generalities on global variational principles in fibered spaces can be found; for basic notions and assertions see also Brajer\v{c}\'{i}k and Krupka \cite{Brajer}, Krupka and Saunders \cite{Handbook}, Takens \cite{Takens}, and Voln\'{a} and Urban \cite{VU}.

\section{Contact forms on embedded submanifolds}
\label{sec:2}
Suppose that $\pi:Y\to X$ is a fibered manifold of dimension $m+1$, $m\geq 0$, over a \textit{one-dimensional} base manifold $X$ (\textit{fibered mechanics}). For any integer $r\geq 0$, $J^rY$ denotes the $r$-jet prolongation of $Y$, and $\pi^r:J^rY\to X$, $\pi^{r,s}:J^rY\to J^sY$, $0\leq s\leq r$, are the \textit{canonical jet projections}. For any open subset $W\subset Y$, we put $W^r=(\pi^{r,0})^{-1}(W)$, an open subset of $J^rY$. An element $J^r_x\gamma$ of $J^rY$ is the $r$-\textit{jet} of section $\gamma$ of $Y$ with \textit{source} $x\in X$ and \textit{target} $y=\gamma(x)$. The mapping $x\to J^r\gamma (x)=J^r_x\gamma$ is the $r$-\textit{jet prolongation} of $\gamma$. The set $J^rY$ is considered with its natural fibered manifold structure: if $(V,\psi)$, $\psi=(t,y^\sigma)$, $1\leq\sigma\leq m$, is a fibered chart on $Y$, the associated chart on $J^rY$ (resp. on $X$), reads $(V^r,\psi^r)$, $\psi^r=(t,y^\sigma,y_1^\sigma, y_2^\sigma,\ldots,y^\sigma_r)$, resp. $(U,\varphi)$, $\varphi=(t)$, where $V^r=(\pi^{r,0})^{-1}(V)$, resp. $U=\pi(V)$. Recall that by definition of an  $r$-jet $J^r_x\gamma\in V^r$ the coordinate $y^\sigma_l(J^r_x\gamma)$ is defined to be $D^l(y^\sigma \gamma\varphi^{-1})(\varphi(x))$.

We denote by $\Omega^r_0 W$ the ring of differentiable functions defined on $W^r\subset J^rY$, and by $\Omega^r_k W$ the $\Omega^r_0 W$-module of differentiable $k$-forms on $W^r$. The exterior algebra of differential forms on $W^r$ is denoted by $\Omega^r W$. For any function $f:W^r\to \R$, we put $hf = f\circ \pi^{r+1,r}$, and with the help of the chart formulas $h dt=dt,\ h dy^\sigma_l=y^\sigma_{l+1}dt$. These formulas define a (global) homomorphism of exterior algebras $h:\Omega^r W\to \Omega^{r+1}W$ called the $\pi$-\textit{horizontalization}. In particular, $h(df) = (df/dt)dt$, where $df/dt$ is the \textit{formal derivative} of function $f$. 

A $1$-form $\rho\in\Omega^r_1W$ is said to be \textit{contact}, if $h\rho = 0$ or, equivalently, if $(J^r\gamma)^*\rho$ vanishes for any section $\gamma$ of $\pi$, defined on an open subset of $X$. In arbitrary fibered chart $(V,\psi)$, $\psi=(t,y^\sigma)$, on $W\subset Y$, every contact $1$-form $\rho$ has an expression $\rho=A_\sigma^l \omega^\sigma_{l}$ (sum through $0\leq l\leq r-1$, $1\leq\sigma\leq m$) for some functions $A_\sigma^l: V^r\to\R$, where $\omega^\sigma_{l}=dy^\sigma_{l}-y^\sigma_{l+1}dt$. 
A basis of $1$-forms on $V^r$, constituted by the forms $dt$, $\omega^\sigma_{l}$, $dy^\sigma_{r}$, where $0\leq l\leq r-1$, is called the \textit{contact basis}. If $(\bar V, \bar\psi)$, $\bar\psi = (\bar t, \bar y^\nu)$, is another fibered chart on $W\subset Y$ such that $V\cap \bar V \neq\emptyset$, then the contact $1$-forms obey the transformation property $\bar\omega^\nu_{l} = \sum_{s=0}^l ({\pa \bar y^\nu_{l}}/{\pa y^\sigma_{s}})\omega^\sigma_{s}$. This means, in particular, that contact  1-forms $\omega^\sigma_l$ locally generate an ideal in the exterior algebra $\Omega^r W$ of differential forms on $W^r$, called the \textit{contact ideal}.

Let $\rho$ be a differential $1$-form on $W^r\subset J^rY$. The pull-back of $\rho$ with respect to the jet projection $\pi^{r+1,r}:J^{r+1}Y\to J^rY$ has a unique decomposition,
\begin{equation}\label{rozklad}
(\pi^{r+1,r})^*\rho = h\rho + p\rho,
\end{equation}
where $h\rho$ (resp. $p\rho$) is a $\pi^{r+1}$-horizontal (resp. contact) $1$-form on $W^{r+1}$. Indeed, if $\rho\in\Omega^r_1W$ is expressed by $\rho = Adt + B_\sigma^l dy^\sigma_{l}$, then it is easily seen that
\begin{equation*}
h\rho = \left( A + \sum_{l=0}^r B_\sigma^l y^\sigma_{l+1} \right) dt,\quad p\rho = \sum_{l=0}^r B_\sigma^l\omega^\sigma_{l}.
\end{equation*}
Decomposition (\ref{rozklad}) can be directly generalized for arbitrary $k$-forms $\rho\in\Omega^r_k W$. We get
\begin{equation}\label{rozklad3}
(\pi^{r+1,r})^*\rho = p_{k-1}\rho + p_k\rho,
\end{equation}
where $p_{k-1}\rho$, resp. $p_{k}\rho$, is the $(k-1)$-\textit{contact}, resp. $k$-\textit{contact} \textit{component} of $\rho$. This formula is the \textit{canonical decomposition} of $\rho$. If $\rho = \kappa\wedge dt + \chi$, where
\begin{equation*}
\begin{aligned}
\kappa &= \frac{1}{(k-1)!}\sum_{l_1,l_2,\ldots, l_{k-1} = 0}^r A_{\sigma_1\sigma_2\ldots\sigma_{k-1}}^{\,l_1\,\, l_2 \,\ldots \,\, l_{k-1}} dy^{\sigma_1}_{l_1}\wedge dy^{\sigma_2}_{l_2} \wedge\ldots\wedge dy^{\sigma_{k-1}}_{l_{k-1}},\\
\chi &= \frac{1}{k!}\sum_{l_1,l_2,\ldots, l_k = 0}^r B_{\sigma_1\sigma_2\ldots\sigma_{k}}^{\,l_1\,\, l_2 \,\ldots \,\, l_{k}} dy^{\sigma_1}_{l_1}\wedge dy^{\sigma_2}_{l_2} \wedge\ldots\wedge dy^{\sigma_{k}}_{l_{k}},
\end{aligned}
\end{equation*}
then
\begin{equation*}
\begin{aligned}
p_{k-1}\rho &= \frac{1}{(k-1)!}\sum_{l_1,l_2,\ldots,l_k=0}^r  \left( A_{\sigma_1\sigma_2\ldots\sigma_{k-1}}^{\,l_1\,\, l_2 \,\ldots \,\, l_{k-1}} + B_{\sigma_1\sigma_2\ldots\sigma_{k}}^{\,l_1\,\, l_2 \,\ldots \,\, l_{k}} y^{\sigma_k}_{l_k + 1}\right) \\
& \cdot\omega^{\sigma_1}_{l_1} \wedge \omega^{\sigma_2}_{l_2}\wedge\ldots \wedge \omega^{\sigma_{k-1}}_{l_{k-1}} \wedge dt,\\
p_{k}\rho &= \frac{1}{k!}\sum_{l_1,l_2,\ldots,l_k=0}^r  B_{\sigma_1\sigma_2\ldots\sigma_{k}}^{\,l_1\,\, l_2 \,\ldots \,\, l_{k}}
\omega^{\sigma_1}_{l_1} \wedge \omega^{\sigma_2}_{l_2} \wedge \ldots \wedge \omega^{\sigma_{k}}_{l_{k}}.
\end{aligned}
\end{equation*}

A $k$-form $\rho\in\Omega^r_kW$ is said to be $k$-\textit{contact}, if it coincides, up to the pull-back, with its $k$-contact component, that is, if $(\pi^{r+1,r})^*\rho = p_{k}\rho$; similarly, $\rho$ is called $(k-1)$-\textit{contact}, if $(\pi^{r+1,r})^*\rho = p_{k-1}\rho$. The transformation property of contact $1$-forms implies that $k$-contact, resp. $(k-1)$-contact, forms on $W^r$ constitute a submodule of the module $\Omega^r_k W$, which we denote by $\Omega^r_{k,c} W$, resp. $\Omega^r_{k-1,c} W$.

Let us remark that every $k$-form is contact for $k\geq 2$. We extend the notion of contactness for differential forms of arbitrary degree as follows. A $k$-form $\rho\in\Omega^r_k W$ is said to be \textit{strongly contact}, if at every point of $Y$ there exist a fibered chart $(V,\psi)$, $\psi=(t,y^\sigma)$, and a $(k-1)$-contact $(k-1)$-form $\eta$ on $V^r$ such that
\begin{equation}\label{strong}
p_k (\rho - d\eta) = 0
\end{equation}
holds on $V^r$. This definition is also motivated by the construction of a variational sequence (cf. Sec. \ref{sec:3}). It is easily seen that a $k$-form 
$\rho\in\Omega^r_k W$ is strongly contact if and only if at every point of $Y$ there exist a fibered chart $(V,\psi)$, $\psi=(t,y^\sigma)$, a $k$-contact $k$-form $\mu$, and a $(k-1)$-contact $(k-1)$-form $\eta$, defined on $V^r$, such that $\rho = \mu + d\eta$. This formula defines a direct sum decomposition over the chart neighborhood $V$,
\begin{equation*}
\Omega^r_k V \ni \rho \to (\mu,\eta)\in \Omega^r_{k,c}V \oplus  \Omega^r_{k-1,c}V.
\end{equation*}
The exterior derivative of a strongly contact form is again strongly contact.

Now we shall consider smooth embeddings into Euclidean spaces which are morphisms of fibered manifolds over \textit{one-dimensional} basis $\R$, and study basic properties of contact forms induced on embedded submanifolds. The canonical fibered coordinates on $\R^{m+1} = \R\times\R^m$ over $\R$ are denoted by $(t,x^\sigma)$, $1\leq\sigma\leq m$, and $(t,x^\sigma,x_1^\sigma, x_2^\sigma,\ldots,x^\sigma_{r})$ are the global associated coordinates on $J^r(\R\times\R^{m})$. As usual, for lower orders we denote the derivative coordinates by the dot sign: if $r=2$, we write $(t,x^\sigma,\dot{x}^\sigma, \ddot{x}^\sigma)$ instead of $(t,x^\sigma,x_1^\sigma, x_2^\sigma)$.

Let $m,n$ be positive integers, $n\leq m$, and let $Q$ be an $n$-dimensional embedded submanifold of the Euclidean space $\R^{m}$. The canonical embedding 
\begin{equation*}
\iota : \R\times Q\to \R\times\R^{m}
\end{equation*}
is a smooth morphism of fibered manifolds over the identity mapping of the real line $\R$. The canonical projections $\pi_Q:\R\times Q\to \R$ and $\pi:\R\times\R^{m}\to \R$ satisfy $\pi\circ\iota = \pi_Q$, and every section $\gamma_Q$ of $\pi_Q$ induces a section $\gamma$ of $\pi$, defined by the composition $\gamma = \iota \circ \gamma_Q$. Indeed, for a section $\gamma_Q$ of $\pi_Q$ we have $\pi\circ\gamma = \pi\circ \iota\circ\gamma_Q = \pi_Q \circ\gamma_Q = \id_\R $. Thus, we may define the $r$\textit{-th jet prolongation} $J^r\iota$ of $\iota$ by the formula
\begin{equation}\label{prolong}
J^r(\R\times Q) \ni J^r_t\gamma_Q \to J^r\iota (J^r_t\gamma_Q) = J^r_t(\iota\circ \gamma_Q) \in J^r(\R\times \R^{m}).
\end{equation}
Mapping (\ref{prolong}) acts on differential forms by means of the \textit{pull-back} operation. For arbitrary $k\geq 0$, we get a homomorphism of modules of differential forms,
\begin{equation}\label{pull}
J^r\iota{}^* : \Omega^r_k (\R\times \R^{m}) \to \Omega^r_k (\R\times Q).
\end{equation}

The following result is a straightforward application of the existence of partition of unity subordinate to an adapted atlas to a given embedded submanifold. In particular, it shows that the pull-back (\ref{pull}) is a surjective mapping.

\begin{lemma}\label{extension}
Let $Q$ be an embedded submanifold of $\R^{m}$, $W\subset \R\times Q$ an open set. Suppose $\eta\in\Omega^r_kW$ be a $k$-form on $W^r\subset J^r(\R\times Q)$. Then there exist a neighbourhood $U$ of $\iota(W)$ in $\R^{m+1}$ and a $k$-form $\tilde\eta\in\Omega^r_kU$ on $U^r\subset J^r\R^{m+1}$ such that $J^r\iota{}^*\tilde\eta = \eta$.
\end{lemma}
\begin{remark}
Note that Lemma \ref{extension} holds also for contact and strongly contact forms. We shall utilize this lemma in the study of variational sequences (Sec. \ref{subsec3}).
\end{remark}

For further considerations we find a chart expression of $J^r\iota$ (\ref{prolong}) by a straightforward calculation.

\begin{lemma}\label{lem:prolong}
Let $(U,\varphi)$, $\varphi = (t, q^i)$, $1\leq i\leq n$, be a fibered chart on $\R\times Q$. Suppose the embedding $\iota: \R\times Q\to \R\times\R^{m}$ is expressed by equations
\begin{equation}\label{iota}
t\circ\iota = t,\quad x^\sigma\circ\iota = f^\sigma (q^i),\quad 1\leq i\leq n,\ 1\leq\sigma\leq m.
\end{equation}
Then the $r$-th prolongation $J^r\iota$ of $\iota$ is expressed in the associated chart $(U^r,\varphi^r)$, $\varphi^r=(t,q^i,q^i_1,q^i_2,\ldots,q^i_{r})$, on $J^r(\R\times Q)$ by the equations
\begin{equation}\label{chartex}
\begin{aligned}
t\circ J^r\iota &= t,\quad x^\sigma\circ J^r\iota = f^\sigma (q^i),\quad
\dot x^\sigma\circ J^r\iota = \frac{d}{dt} f^\sigma (q^i),\\
\ddot x^\sigma\circ J^r\iota &=\frac{d^2}{dt^2} f^\sigma (q^i),\quad
x^\sigma_{l}\circ J^r\iota  = \frac{d^l}{dt^l}f^\sigma (q^i),
\end{aligned}
\end{equation}
where $0\leq l\leq r$.
\end{lemma}

The following lemma describes properties of contact forms with respect to the pull-back $J^r\iota{}^*$ (\ref{pull}). In particular, we show that the property of contactness and strongly contactness is preserved along a jet prolongation of an embedding.

\begin{lemma}\label{lem:contact2}

{\rm (a)} If a $1$-form $\rho\in\Omega^r_1 R^{m+1}$ is contact, then $J^r\iota{}^*\rho\in\Omega^r_1 (\R\times Q)$ is contact.

{\rm (b)} The pull-back $J^r\iota{}^*\rho$ of any strongly contact $k$-form $\rho$ on $J^r\R^{m+1}$ is strongly contact.

{\rm (c)} Suppose the embedding $\iota:\R\times Q\to\R^{m+1}$ is expressed with respect to a fibered chart $(U,\varphi)$, $\varphi=(t,q^i)$, on $S$ by equations {\rm (\ref{iota})}. Then pull-backs of contact $1$-forms $\omega^\sigma_{l}$, $0\leq l\leq r-1$, on $J^r\R^{m+1}$ with respect to $J^r\iota$ read
\begin{equation}\label{pullcont}
\begin{aligned}
J^r\iota{}^*\omega^\sigma &= \frac{\pa f^\sigma}{\pa q^i}\eta^i,\\
J^r\iota{}^*\dot\omega^\sigma &=\frac{d}{dt}\left( \frac{\pa f^\sigma}{\pa q^i} \right)\eta^i +\frac{\pa f^\sigma}{\pa q^i}\dot\eta^i,\\
J^r\iota{}^*\ddot\omega^\sigma &=\frac{d^2}{dt^2}\left( \frac{\pa f^\sigma}{\pa q^i} \right)\eta^i + 2 \frac{d}{dt}\left( \frac{\pa f^\sigma}{\pa q^i} \right)\dot\eta^i +\frac{\pa f^\sigma}{\pa q^i}\ddot\eta^i,\\
J^r\iota{}^*\omega^\sigma_{l} &= \sum_{s=0}^l \frac{\pa}{\pa q^i_{s}} \left( \frac{d^l f^\sigma}{dt^l}  \right) \eta^i_{s},\quad 0\leq l\leq r-1,
\end{aligned}
\end{equation}
where $\eta^i_{s} = dq^i_{s} - q^i_{s+1} dt$, $0\leq s\leq r-1$, $1\leq i\leq n$, are contact forms on $U^r\subset J^r(\R\times Q)$.
\end{lemma}
\begin{proof}
(a) Suppose $\rho$ is a contact $1$-form on $J^r\R^{m+1}$, that is by definition, the pull-back $J^r\gamma{}^*\rho$ vanishes for every smooth section $\gamma$ of $\pi:\R^{m+1}\to\R$. Let $\gamma_Q$ be a section of $\pi_Q:\R\times Q\to\R$. We compute the pull-back $J^r\gamma_Q{}^* J^r\iota{}^*\rho$. For every point $t$ from the domain of $\gamma_Q$, $(J^r\iota \circ J^r\gamma_Q)(t) = J^r\iota (J^r_t\gamma_Q) = J^r_t (\iota\circ\gamma_Q) = J^r(\iota\circ\gamma_Q)(t)$, hence we get
\begin{equation*}
J^r\gamma_Q{}^* J^r\iota{}^*\rho = (J^r\iota \circ J^r\gamma_Q)^*\rho = J^r(\iota\circ\gamma_Q)^*\rho.
\end{equation*}
Since $\iota\circ\gamma_Q$ is a section of $\R^{m+1}$, by the assumption on contactness of $\rho$ this expression vanishes hence $1$-form $J^r\iota{}^*\rho$ on $J^r(\R\times Q)$ is a also contact, as required.

(b) For $k=1$ the assertion is equivalent with (a). Let $k\geq 2$, and suppose $\rho\in\Omega^r_k\R^{m+1}$ is decomposed as $\rho = \mu + d\eta$, where $\mu$, resp. $\eta$, is a $k$-contact $k$-form, resp. $(k-1)$-contact $(k-1)$-form. Then $J^r\iota{}^*\rho = J^r\iota{}^*\mu + d J^r\iota{}^*\eta$. Since a pull-back of wedge products equals wedge products of pull-backs, we obtain by the assertion (a) that also $J^r\iota{}^*\mu$, resp. $J^r\iota{}^*\eta$, is a $k$-contact $k$-form, resp. $(k-1)$-contact $(k-1)$-form, on $J^r(\R\times Q)$. This means, however, that $J^r\iota{}^*\rho\in\Omega^r_k(\R\times Q)$ is strongly contact.

(c) Formulas (\ref{pullcont}) are obtained by a staighforward calculation.
\end{proof}

\begin{remark}[Adapted charts]\label{adapted}
In order to simplify calculations and proofs, we may take advantage of adapted charts to submanifolds. Let $Q$ be an embedded $n$-dimensional submanifold of $Y$, where $\dim Y = m$, $n\leq m$. 
Recall that a chart $(V,\psi)$, $\psi = (u^\sigma)$, on $Y$ is \textit{adapted} to $Q$, if there exist constants $c^{n+1},c^{n+2},\ldots,c^{m}\in\R$ such that the set $Q$ is expressed in $(V,\psi)$ by equations $u^\kappa\circ\iota = c^\kappa$, $n+1\leq \kappa\leq m$, where $\iota:Q\to Y$ is the canonical smooth embedding. In this case, the set $Q\cap V$ is also called an $n$-\textit{dimensional slice} in an open subset $V$ of $Y$. Without loss of generality, we may restrict ourselves to adapted charts to submanifold $Q$ which are \textit{rectangle} hence the canonical embedding $\iota$ has equations
\begin{equation}\label{adp}
\begin{aligned}
u^j\circ\iota &= u^j|_{V\cap Q},\ 1\leq j\leq n,\quad
u^\kappa\circ\iota = 0,\ n+1\leq \kappa\leq m.
\end{aligned}
\end{equation}
\end{remark}
If $(V,\psi)$, $\psi=(u^\sigma)$, is a chart on $\R^{m}$, adapted to $n$-dimensional submanifold $Q$ of $\R^{m}$, and satisfying (\ref{adp}), one can easily observe that the chart expressions (\ref{chartex}) of Lemma \ref{lem:prolong}, resp. (\ref{pullcont}) of Lemma \ref{lem:contact2}, with respect to the associated chart on $Q$ reduce to \textit{non}-trivial expressions for $q^j_{l}\circ J^r\iota$, resp. $J^r\iota{}^*\omega^j_{l}$, only.

\section{Variational sequences}
\label{sec:3}
\subsection{Variational sequences over fibered manifolds}
\label{sec:seq}
Let us recall basic results. Suppose $Y$ to be a fibered manifold over a one-dimensional base $X$, where $\dim Y = m+1$, $m\geq 0$. Let $k\geq 0$, $r\geq 0$, be integers. Consider the sheaf of smooth $k$-forms over $J^rY$ and its direct image with respect to the jet projection $\pi^{r,0}:J^rY\to Y$, denoted by $\Omega^r_k$. Similarly, $\Omega^r_{k,c}$ (resp. $\Omega^r_{k-1,c}$) is the direct image of the sheaf of $k$-contact $k$-forms (resp. $(k-1)$-contact $(k-1)$-forms) over $J^rY$ with respect to $\pi^{r,0}$.  We let denote $\Omega^r_{0,c}=\{ 0\}$, $\Omega^r_{k,c}=\ker p_{k-1}$, and $\Theta^r_k = \Omega^r_{k,c} + d \Omega^r_{k-1,c}$, $k\geq 1$, where $p_{k-1}:\Omega^r_k\to \Omega^{r+1}_{k-1,c}$ is the sheaf space morphism assigning to a $k$-form its $(k-1)$-contact component (see Sec. 2, (\ref{rozklad3})), and $d \Omega^r_{k-1,c}$ is the image sheaf of $\Omega^r_{k-1,c}$ with respect to the exterior derivative morphism $d$.

For arbitrary open subset $W$ of $Y$, we have $\Omega^r_kW$ (resp. $\Omega^r_{k,c}W$) the Abelian group of differential $k$-forms (resp. $k$-contact $k$-forms), defined on an open subset $W^r=(\pi^{r,0})^{-1}(W)$ of the $r$-th jet prolongation $J^rY$. $d \Omega^r_{k-1,c}W$ is the Abelian group of $k$-forms on $W^r$, locally expressed as exterior derivatives of $(k-1)$-contact $(k-1)$-forms, and $\Theta^r_k W$ is a subgroup (but \textit{not} a submodule) of $\Omega^r_kW$, elements of which are the strongly contact $k$-forms on $W^r$ (see Sec. 2, (\ref{strong})). 

From decomposition properties of strongly contact forms, we observe that the exterior derivative operator $d$ defines a subsequence
\begin{equation}\label{subsequence}
\begin{aligned}
& 0\to\Theta^r_1 \to \Theta^r_2 \to\cdots \to \Theta^r_M  \to0,
\end{aligned}
\end{equation}
of the \textit{de Rham sequence} $0\to\R\to\Omega^r_0\to\Omega^r_1\to\Omega^r_2\to\ldots\to\Omega^r_N \to 0$, where $M=mr+1$, $N=\dim J^r Y = m(r+1) +1$. (\ref{subsequence}) is the \textit{contact subsequence}, and the quotient sequence
\begin{equation}\label{quotient}
\begin{aligned}
&0\to\R_Y\to\Omega^r_0 \to\Omega^r_1  / \Theta^r_1 \to \Omega^r_2  / \Theta^r_2 
\to \cdots\\
&\qquad \to \Omega^r_M  / \Theta^r_M  \to \Omega^r_{M+1}  \to \cdots \to \Omega^r_N  \to 0,
\end{aligned}
\end{equation}
briefly denoted by $0\to\R_Y\to\V^r_Y$, is the \textit{variational sequence of order $r$ over} $Y$. An element of $\Omega^r_k / \Theta^r_k$, represented by a $k$-form $\rho\in\Omega^r_kW$, is denoted by $[\rho]$. Morphisms $E_k : \Omega^r_k / \Theta^r_k \to \Omega^r_{k+1} / \Theta^r_{k+1}$ in (\ref{quotient}) are defined by
\begin{equation}\label{quodef}
E_k ([\rho]) = [d\rho].
\end{equation}
Summarizing, we have the commutative diagram
\begin{equation*}
\begin{diagram}
&&&&&&0&&0&&0\\
&&&&&&\uTo&&\uTo&&\uTo\\
&&&&&&\Omega^r_1/\Theta^r_1&\rTo&\Omega^r_2/\Theta^r_2&\rTo&\Omega^r_3/\Theta^r_3&\rTo&\cdots\\
&&&&&\ruTo&\uTo&&\uTo&&\uTo\\
0&\rTo&\R&\rTo&\Omega^r_0&\rTo&\Omega^r_1&\rTo&\Omega^r_2&\rTo&\Omega^r_3&\rTo&\cdots\\
&&&&\uTo&&\uTo&&\uTo&&\uTo\\
&&&&0&\rTo&\Theta^r_1&\rTo&\Theta^r_2&\rTo&\Theta^r_3&\rTo&\dots\\
&&&&&&\uTo&&\uTo&&\uTo\\
&&&&&&0&&0&&0
\end{diagram}
\end{equation*}

Now we describe fundamental properties of the variational sequence theory.

\begin{theorem}\label{th1}
The contact subsequence {\rm (\ref{subsequence})} is exact.
\end{theorem}

Since exactness of a quotient sequence follows from exactness of its subsequence of Abelian groups, Theorem \ref{th1} immediately implies that the variational sequence is also \textit{exact} hence $0\to\R_Y\to\V^r_Y$ (\ref{quotient}) constitutes a resolution of the constant sheaf $\R_Y$.

The next result plays an important role in the study of \textit{order reduction} of variational objects, embeddings of $r$-th order variational sequences into sequences of order $r+1$, and it serves as the main tool for finding suitable local representatives of classes in the variational sequence. Since the quotient groups are determined up to an isomorphism, we are allowed to consider the classes in the variational sequence (\ref{quotient}) as elements of different sheaves.

\begin{theorem}\label{th2}
The quotient mapping $\Omega^r_kW/\Theta^r_kW \to \Omega^{r+1}_kW/\Theta^{r+1}_kW$, corresponding to the pull-back $\Omega^r_kW\ni\rho\to(\pi^{r+1,r})^*\rho\in \Omega^{r+1}_kW$, is injective.
\end{theorem}

Recall that a sheaf $S$ over a topological space $X$ is called \textit{soft}, if every continuous section of $S$, defined on a closed subset of $X$, can be prolonged to a global continuous section, and a sheaf $S$ over $X$ is called \textit{fine}, if to every locally finite open covering of $X$ there exists a subordinate sheaf partition of unity of the sheaf space germ$\,S$. Every fine sheaf over a paracompact Hausdorff space is soft.

\begin{theorem}\label{fine}
Every of the sheaves $\Omega^r_k$ is fine, and hence soft.
\end{theorem}

Now, as a consequence of Theorem \ref{fine}, we get a theorem on global properties of the variational sequence. Recall that a resolution of an Abelian sheaf $S$ over $X$, $0\to S \to S^0 \to S^1 \to S^2 \to \ldots$, is called \textit{acyclic}, if $H^k(X,S^i)=0$ for all $i\geq 1$. In particular, every soft sheaf $S$ over $X$ satisfies $H^k(X,S)=0$, $k\geq 1$.

\begin{theorem}\label{th4}
The variational sequence of order $r$ over $Y$ {\rm (\ref{quotient})} is an acyclic resolution of the constant sheaf $\ \R_Y$ over $Y$.
\end{theorem}

Denote by $\Gamma\V^r_Y$ the \textit{cochain complex} of Abelian groups of \textit{global sections},
\begin{equation}\label{global}
\begin{aligned}
&0\to\Gamma(Y,R_Y)\to\Gamma(Y,\Omega^r_0) \to\Gamma(Y,\Omega^r_1  / \Theta^r_1) \to \Gamma(Y,\Omega^r_2 / \Theta^r_2)\\
& \to \cdots \to \Gamma(Y,\Omega^r_M  / \Theta^r_M)  \to \Gamma(Y,\Omega^r_{M+1})  \to \cdots \to \Gamma(Y,\Omega^r_N)  \to 0,
\end{aligned}
\end{equation}
and by $H^k(\Gamma\V^r_Y)$ the $k$-th cohomological group of (\ref{global}), this means $H^k(\Gamma\V^r_Y) =  \ker\ E_k / \im\ E_{k-1}$. The following result is a consequence of Theorem \ref{th4} and the well-known \textit{abstract de Rham theorem} (cf. Warner \cite{Warner}).

\begin{theorem}\label{th5}
The cohomology of the complex {\rm (\ref{global})} and the de Rham cohomology of the underlying manifold $Y$ coincide, i.e. for every $k\geq 0$
\begin{equation}\label{cohomeq}
H^k(\Gamma\V^r_Y) = H^k_{\mathrm{deR}}Y.
\end{equation}
\end{theorem}

\subsection{Variational sequences induced by embedded submanifolds}
\label{subsec3}
Let $Q$ be an embedded submanifold of $\R^{m}$, and $\iota:\R\times Q\to\R\times\R^{m}$ be the canonical embedding of fibered manifolds over $\R$, locally expressed by equations $t\circ\iota = t$, $x^\sigma\circ\iota = f^\sigma(q^i)$ (Sec. 2, (\ref{iota})), with respect to the global canonical coordinates $(t,x^\sigma)$ on $\R\times\R^m$, and a fibered chart $(U,\varphi)$, $\varphi=(t,q^i)$, on $\R\times Q$. The $r$-th jet prolongation $J^r\iota:J^r(\R\times Q)\to J^r(\R\times\R^m)$ (Sec. \ref{sec:2}, (\ref{prolong})) has a chart expression described by Lemma \ref{lem:prolong}.

Let us study variational sequences over both fibered manifolds $\R\times Q$ and $\R\times\R^{m}$. To this purpose we introduce a pull-back of presheaves of forms with respect to an embedding. Denote by $\Omega^r_k$, resp. ${}^{Q}\!\Omega^r_k$, the direct image of the sheaf of smooth $k$-forms over $J^r(\R\times\R^{m})$, resp. $J^r(\R\times Q)$, with respect to the jet projection $\pi^{r,0}:J^r(\R\times\R^{m})\to \R\times\R^{m}$, resp. $\pi^{r,0}_Q:J^r(\R\times Q)\to \R\times Q$. 
It should be pointed out that $\Omega^r_k$ is a simplified notation for the sheaf ${}^{\R^{m}\!\!}\Omega^r_k$.
Analogously to Sec. \ref{sec:seq}, we denote the sheaves of strongly contact forms by $\Theta^r_k = \Omega^r_{k,c} + d \Omega^r_{k-1,c}$, and ${}^{Q}\!\Theta^r_k = {}^{Q}\!\Omega^r_{k,c} + d {}^{Q}\!\Omega^r_{k-1,c}$. An element of $\Omega^r_k / \Theta^r_k$, represented by a $k$-form $\rho\in \Omega^r_kU$, is denoted by $[\rho]$; an element of ${}^{Q}\!\Omega^r_k / {}^{Q}\!\Theta^r_k$, represented by a $k$-form $\tau\in {}^{Q}\!\Omega^r_kW$, is denoted by $[\tau]_Q$. 

The following constructions allow us to introduce variational sequences on embedded submanifolds. For arbitrary open set $W\subset \R\times Q$ we put   
\begin{equation}\label{pull1}
\begin{aligned}
J^r\iota{}^*(\Omega^r_k)(W) &= \{ J^r\iota{}^*\rho|_{W^r}\, |\, \rho\in \Omega^r_k U\},\\
J^r\iota{}^*(\Theta^r_k)(W) &= \{ J^r\iota{}^*\rho|_{W^r}\, |\, \rho\in \Theta^r_k U \},
\end{aligned}
\end{equation}
and
\begin{equation}\label{pullback}
\begin{aligned}
J^r\iota{}^*(\Omega^r_k / \Theta^r_k)(W) &= \{ [J^r\iota{}^*\rho |_{W^r}]_Q \, |\, \rho\in \Omega^r_k U \},
\end{aligned}
\end{equation}
where $U$ runs over neighbourhoods of $\iota(W)$. 
From definition (\ref{pull1}) it is straightforward that the correspondence assigning to an open set $W\subset \R\times Q$ the Abelian group $J^r\iota{}^*(\Omega^r_k)(W)$, resp. $J^r\iota{}^*(\Theta^r_k)(W)$, is a presheaf, and $J^r\iota{}^*(\Omega^r_k)(W)$, resp. $J^r\iota{}^*(\Theta^r_k)(W)$, is an Abelian subgroup of ${}^{Q}\!\Omega^r_k W $, resp. ${}^{Q}\!\Theta^r_k W $, for every open set $W$. The next lemma describes the presheaf  {\rm (\ref{pullback})}.

\begin{theorem}\label{lemRov} The presheaf $J^r\iota{}^*(\Omega^r_k / \Theta^r_k)$ is a sheaf, and it coincides with the quotient sheaf ${}^{Q}\!\Omega^r_k / {}^{Q}\!\Theta^r_k$,
\begin{equation*}
J^r\iota{}^*(\Omega^r_k / \Theta^r_k)
= {}^{Q}\!\Omega^r_k / {}^{Q}\!\Theta^r_k.
\end{equation*}
\end{theorem}
\begin{proof}
By Lemma \ref{extension} we have $J^r\iota{}^*(\Omega^r_k)(W) = {}^{Q}\!\Omega^r_k W$, $J^r\iota{}^*(\Theta^r_k)(W) = {}^{Q}\!\Theta^r_k W$, for every open set $W\subset \R\times Q$, and thus $J^r\iota{}^*(\Omega^r_k)$, resp. $J^r\iota{}^*(\Theta^r_k)$, coincides with the sheaf ${}^{Q}\!\Omega^r_k$, resp. ${}^{Q}\!\Theta^r_k$. From (\ref{pullback}), however, it follows that $J^r\iota{}^*(\Omega^r_k / \Theta^r_k)(W)$ is the quotient group $J^r\iota{}^*(\Omega^r_k)(W) /  J^r\iota{}^*(\Theta^r_k)(W)$. Hence we obtain the required identity.
\end{proof}

Consider a sheaf morphism
$
{J^r\iota}^* : \Omega^r_k / \Theta^r_k \to 
{}^{Q}\!\Omega^r_k / {}^{Q}\!\Theta^r_k,
$
defined as follows. For any element $[\rho]$ of $\Omega^r_k / \Theta^r_k$, where $\rho\in \Omega^r_k U$, $U$ is an open subset of $\R^{m+1}$, we put
\begin{equation}\label{class}
{J^r\iota}^*([\rho]) = [{J^r\iota}^*\rho]_Q.
\end{equation}
It is easy to see that definition (\ref{class}) does \textit{not} depend on the choice of a representative $\rho$ of class the $[\rho]$. Indeed, if $\rho'$ is another representative of $[\rho]$, then $\rho-\rho'$ belongs to $\Theta^r_k$. However, by Lemma \ref{lem:contact2}, (b), it follows that ${J^r\iota}^*(\rho-\rho')\in {}^{Q}\!\Theta^r_k$ hence we get
\begin{equation*}
[{J^r\iota}^*\rho]_Q = [{J^r\iota}^*(\rho-\rho')]_Q + [{J^r\iota}^*\eta]_Q = [{J^r\iota}^*\eta]_Q .
\end{equation*}

The variational sequence morphisms $E_k : \Omega^r_k / \Theta^r_k \to \Omega^r_{k+1}/ \Theta^r_{k+1}$, and ${}^{Q}\!E_k : {}^{Q}\!\Omega^r_k / {}^{Q}\!\Theta^r_k \to {}^{Q}\!\Omega^r_{k+1}/ {}^{Q}\!\Theta^r_{k+1}$, are defined by $E_k([\rho]) = [d\rho]$, and ${}^{Q}\! E_k([\tau]_Q) = [d\tau]_Q$, respectively (Sec. \ref{sec:seq}, (\ref{quodef})). We have the diagram
\begin{equation}\label{diagram2}
\begin{diagram}
\cdots & \rTo^{{}^{Q}\!E_{k-1}} & {}^{Q}\!\Omega^r_k/ {}^{Q}\!\Theta^r_k&\rTo^{{}^{Q}\!E_{k\ }} & {}^{Q}\!\Omega^r_{k+1}/ {}^{Q}\!\Theta^r_{k+1}&\rTo^{{}^{Q}\!E_{k+1\ }}&\cdots\\
& &\uTo_{{J^r\iota}^*}&&\uTo_{{J^r\iota}^*}\\
\cdots &\rTo^{E_{k-1\ }} & \Omega^r_k / \Theta^r_k &\rTo^{E_{k\ \ }} & \Omega^r_{k+1} / \Theta^r_{k+1} &\rTo^{E_{k+1\ }}&\cdots
\end{diagram}
\end{equation}
where by Lemma \ref{lemRov} the morphisms ${J^r\iota}^*$ (\ref{class}) are surjective. The following theorem shows that squares in (\ref{diagram2}) commute.

\begin{theorem}\label{th6}
For every positive integer $k$, the morphisms $E_{k}$ in the variational sequence $0\to\R_{\R^{m}}\to\V^r_{\R^{m}}$ over $\R\times\R^{m}$, and ${}^{Q}\!E_{k}$ in the variational sequence $0\to\R_Q\to\V^r_Q$ over $\R\times Q$ satisfy the identity
\begin{equation}\label{comm}
{}^{Q}\!E_{k}\circ {J^r\iota}^* = {J^r\iota}^* \circ E_{k}.
\end{equation}
\end{theorem}
\begin{proof}
Identity (\ref{comm}) follows immediately from definitions. Indeed, applying formula (\ref{class}) for ${J^r\iota}^*$ and the definition of morphisms $E_k$, we get for any element $\rho\in \Omega^r_k U$,
\begin{equation*}
\begin{aligned}
({}^{Q}\! E_{k}\circ {J^r\iota}^*)([\rho]) &= {}^{Q}\!E_{k}([{J^r\iota}^*\rho]_Q)=[d {J^r\iota}^*\rho]_Q = [{J^r\iota}^* d\rho]_Q \\
&= {J^r\iota}^*([d\rho]) = ({J^r\iota}^*\circ E_{k})([\rho]),
\end{aligned}
\end{equation*} 
as required.
\end{proof}

\begin{remark}
Theorems \ref{lemRov} and \ref{th6} say that the variational sequence $0\to\R_Q\to\V^r_Q$ over $\R\times Q$ is completely described by the embedded submanifold $Q$ of $\R^{m}$ and the variational sequence $0\to\R_{\R^{m}}\to\V^r_{\R^{m}}$ over $\R\times\R^{m}$.
\end{remark}

Denote by $\Gamma \V^r_{\R^{m}}$ and $\Gamma \V^r_{Q}$, the cochain complexes of global sections associated with the variational sequences $0\to\R_{\R^{m}}\to\V^r_{\R^{m}}$ and $0\to\R_Q\to\V^r_Q$, respectively (cf. (\ref{global})). Theorem \ref{th6} allows us to introduce a cochain mapping, consisting of a family of homomorphisms ${J^r\iota}^* : \Gamma (\R\times\R^{m}, \Omega^r_k / \Theta^r_k) \to \Gamma (\R\times Q, {}^{Q}\!\Omega^r_k / {}^{Q}\!\Theta^r_k)$, defined by means of (\ref{class}). The $k$-th cohomological group $H^k(\Gamma \V^r_{\R^{m}})$, resp. $H^k(\Gamma \V^r_{Q})$, of the complex $\Gamma \V^r_{\R^{m}}$, resp. $\Gamma \V^r_{Q}$, is defined in the standard sense (cf. Sec. \ref{sec:seq}). We get the induced homomorphism of cohomological groups $H^k_{\R^{m}, Q}: H^k(\Gamma \V^r_{\R^{m}}) \to H^k(\Gamma \V^r_{Q})$, defined by $H^k_{\R^{m},Q} ([\rho]) = [{J^r\iota}^*\rho]$ (cf. Warner \cite{Warner}). The abstract de Rham theorem for cochain complexes of global sections of variational sequences, see Thereom \ref{th5}, (\ref{cohomeq}), implies that $H^k(\Gamma \V^r_{\R^{m}}) = H^k_{\mathrm{deR}}(\R\times\R^{m}) = 0$ and $H^k(\Gamma \V^r_{Q}) = H^k_{\mathrm{deR}}(\R\times Q) = H^k_{\mathrm{deR}}Q$, hence the homomorphism $H^k_{\R^{m}, Q}: 0 \to H^k_{\mathrm{deR}}Q$ is a trivial injection.

\subsection{Applications: The calculus of variations}
\label{calvar}
In this section we summarize basic higher-order variational formulas characterizing the Euler--Lagrange and Helmholtz mappings of the calculus of variations. One can easily specify these formulas for first- and second-order problems important for applications.

Let $(V,\psi)$, $\psi=(t,y^\sigma)$, be a fibered chart on $W\subset Y$, and the associated chart on $W^r\subset J^rY$ is denoted by $(V^r,\psi^r)$, $\psi^r=(t,y^\sigma,y^\sigma_1,y^\sigma_2,\ldots, y^\sigma_r)$. If $\rho$ is a $1$-form on $W^r\subset J^rY$, then its class $[\rho]$ is an element of $\Omega^r_1 W /\Theta^r_1 W$, represented by a differential $1$-form $\lambda$ on $W^{r+1}\subset J^{r+1}Y$, with chart expression
\begin{equation}\label{lagrangian}
\lambda = \mathscr{L} dt,
\end{equation}
where $\mathscr{L}: V^{r+1}\to\R$ is a real-valued function. $\lambda$ is the \textit{Lagrangian} associated with $\rho$, and the function $\mathscr{L}$ is a (local) \textit{Lagrange function} associated with $\lambda$. 

If $\rho$ is a $2$-form on $W^r$, then its class $[\rho]$ is an element of $\Omega^r_2 W /\Theta^r_2 W$, represented by a differential $2$-form $\varepsilon$ on $W^{s}\subset J^{s}Y$ for some $s\leq 2r+1$, which has a chart expression
\begin{equation}\label{sourceform}
\varepsilon = \varepsilon_\sigma \omega^\sigma\wedge dt,
\end{equation}
where $\varepsilon_\sigma: V^s \to\R$, $\varepsilon_\sigma = \varepsilon_\sigma (t, y^\sigma, y^\sigma_1, y^\sigma_2, \ldots, y^\sigma_s)$. A differential $2$-form locally expressed by (\ref{sourceform}) is a \textit{source form} (of order $s$), associated with $\rho$. Every Lagrangian $\lambda$ defines a source form $E_\lambda$, locally expressed by 
\begin{equation}\label{ELform}
E_\lambda = \E_\sigma(\mathscr{L}) \omega^\sigma \wedge dt,
\end{equation}
where
\begin{equation}\label{EulerLagrange}
\E_\sigma (\mathscr{L}) = \frac{\pa \mathscr{L}}{\pa y^\sigma} + \sum_{l=1}^r (-1)^l \frac{d^l}{dt^l} \frac{\pa\mathscr{L}}{\pa y^\sigma_l}.
\end{equation}
The coefficients of $E_\lambda$ are the \textit{Euler--Lagrange expressions}, associated with $\mathscr{L}$. $E_\lambda$ is called the \textit{Euler--Lagrange form} (or \textit{Euler--Lagrange class}), associated with $\lambda$. 

The class $[\rho]$ of a $3$-form $\rho$ on $W^r$ belongs to the quotient  $\Omega^r_3 W /\Theta^r_3 W$, and is represented by a differential $3$-form $\vartheta$ on $W^{s}\subset J^{s}Y$ for some $s\leq 4r+3$, which is locally expressed by
\begin{equation}\label{source2}
\vartheta = \frac{1}{2}\sum_{l=0}^r\vartheta^l_{\nu\sigma} \omega^\nu_l \wedge \omega^\sigma\wedge dt,
\end{equation}
where the components $\vartheta^l_{\nu\sigma}$ are functions of $t, y^\kappa, y^\kappa_1, y^\kappa_2, \ldots, y^\kappa_s$. Every form $\varepsilon$ (\ref{sourceform}) defines a class $H_\varepsilon$ in $\Omega^r_3 W / \Theta^r_3 W$ by 
\begin{equation}\label{Hform}
H_\varepsilon = \H^l_{\nu\sigma}(\varepsilon_\kappa) \omega^\nu_l \wedge \omega^\sigma\wedge dt,
\end{equation}
where
\begin{equation}\label{Helmholtz}
\H^l_{\nu\sigma}(\varepsilon_\kappa) = \frac{\pa \varepsilon_\sigma}{\pa y^\nu_l} - (-1)^l \frac{\pa \varepsilon_\nu}{\pa y^\sigma_l} - \sum_{s=l+1}^r (-1)^s {s \choose l} \frac{d^{s-l}}{dt^{s-l}} \frac{\pa \varepsilon_\nu}{\pa y^\sigma_s}.
\end{equation}
Then $H_\varepsilon$ is the \textit{Helmholtz form} (\textit{Helmholtz class}), associated with source form $\varepsilon$, and the components $\H^l_{\nu\sigma}(\varepsilon_\kappa)$ are the \textit{Helmholtz expressions}.

The $\R$-linear mapping $\lambda \to E_\lambda$, assigning to a Lagrangian its Euler--Lagrange form, is the \textit{Euler--Lagrange mapping}, and $\varepsilon \to H_\varepsilon$, assigning to a source form its Helmholtz form, we call the \textit{Helmholtz mapping}. A Lagrangian $\lambda$ which belongs to the \textit{kernel} of the Euler--Lagrange mapping, that is $E_\lambda = 0$, is said to be (\textit{locally}) \textit{trivial}, or \textit{null}. A source form $\varepsilon$ is said to be (\textit{locally}) \textit{variational}, if it belongs to the \textit{image} of the Euler--Lagrange mapping; this means that there exists a Lagrangian $\lambda$ such that $\varepsilon = E_\lambda$, i.e. in components, $\varepsilon_\sigma$ coincide with the Euler--Lagrange expressions associated with a Lagrange function $\mathscr{L}$. In addition, if the Lagrangian $\lambda$ is defined on $J^{r+1}Y$, $\varepsilon$ is \textit{globally} \textit{variational}. The problem to describe the kernel and the image of the Euler--Lagrange mapping belongs to the \textit{inverse problem} of the calculus of variations.

The following theorem explains the relationship between the variational sequence and the classical calculus of variations, and enables us to study \textit{global} properties of the Euler--Lagrange and Helmholtz mappings.

\begin{theorem}\label{th3}
Let $W$ be an open subset in fibered manifold $Y$, and $(V,\psi)$, $\psi=(t,q^\sigma)$, be a fibered chart on $Y$ such that $V\subset W$.

{\rm (a)} Let $\lambda$ be a Lagrangian associated with a $1$-form $\rho\in\Omega^r_1 W$, given by {\rm (\ref{lagrangian})}. Then $E(\lambda)$ has an expression
\begin{equation*}
E(\lambda) = E_\sigma([d\rho]) \omega^\sigma\wedge dt,
\end{equation*}
where its coefficients coincide with the Euler--Lagrange expressions {\em (\ref{EulerLagrange})},
\begin{equation*}
E_\sigma([d\rho]) = \E_\sigma (\mathscr{L}).
\end{equation*}

{\rm (b)} Let $\varepsilon$ be a source form associated with a $2$-form $\rho\in\Omega^r_2 W$, given by {\rm (\ref{source})}. Then $E(\varepsilon)$ has an expression
\begin{equation*}
E(\varepsilon) = \frac{1}{2}\sum_{l=0}^r E^l_{\nu\sigma}([d\rho]) \omega^\nu_l \wedge\omega^\sigma\wedge dt,
\end{equation*}
where its coefficients coincide with the Helmholtz expressions {\em (\ref{Helmholtz})},
\begin{equation*}
E^l_{\nu\sigma}([d\rho]) = \H^l_{\nu\sigma}(\varepsilon_\tau).
\end{equation*}
\end{theorem}

\begin{corollary}\label{corglobal}

{\rm (a)}
Let $\varepsilon$ be a source form on $J^r Y$. If $\varepsilon$ is locally variational and the second de Rham cohomology group of $Y$ vanishes, i.e. $H^2_{\mathrm{deR}} Y = 0$, then $\varepsilon$ is globally variational.

{\rm (b)}
Let $\rho$ be a $1$-form on $J^r Y$. If $\rho$ is locally variationally trivial and the first de Rham cohomology group of $Y$ vanishes, i.e. $H^1_{\mathrm{deR}} Y = 0$, then $\rho$ is globally variationally trivial.
\end{corollary}

Consider now again the variational sequence $0\to\R_{\R^{m}}\to\V^r_{\R^{m}}$ over the Euclidean space $\R\times\R^{m}$ and an embedded submanifold $Q$ of $\R^m$. As we observed, these data already determine the induced variational sequence $0\to\R_{Q}\to\V^r_{Q}$ over $\R\times Q$. Since both $H^2_{\mathrm{deR}} \R^{m}$ and $H^1_{\mathrm{deR}} \R^{m}$ are trivial, we obtain by Corollary 1 that every locally variational source form on $J^r (\R\times\R^{m})$ is automatically globally variational, and every locally variationally trivial $1$-form on $J^r (\R\times\R^{m})$ is globally variationally trivial. In general, an analogous assertion is \textit{not} valid when $\R^{m}$ is replaced with $Q$. 

It is now straightforward to describe the local structure of the Euler--Lagrange and the Helmholtz forms in the induced variational sequence. In the following theorem we consider the embedded submanifold $Q$ and its equations expressed by (\ref{iota}),  
$t\circ\iota = t,\ x^\sigma\circ\iota = f^\sigma (q^i)$, $1\leq i\leq n,\ 1\leq\sigma\leq m$. $E_\lambda$ (resp. $H_\varepsilon$) is the Euler--Lagrange (resp. Helmholtz) form expressed by (\ref{ELform}) (resp. (\ref{Hform})).

\begin{theorem}\label{th7}

{\rm (a)} 
The Euler--Lagrange form $E_{{J^r\iota}^*\lambda}$, associated with ${J^r\iota}^*\lambda$, is expressed by
\begin{equation}\label{constrainedEL}
E_{{J^r\iota}^*\lambda} = (\E_\sigma(\mathscr{L})\circ J^r\iota)\frac{\partial f^\sigma }{\partial q^i} \eta^i\wedge dt,
\end{equation}
where $\eta^i = dq^i - \dot{q}^i dt$ are contact 1-forms on $U^1\subset J^1 (\R\times Q)$.

{\rm (b)}  
The Helmholtz form $H_{{J^r\iota}^*\varepsilon}$, associated with ${J^r\iota}^*\varepsilon$, is expressed by
\begin{equation}\label{constrainedH}
H_{{J^r\iota}^*\varepsilon} = \frac{1}{2}\sum_{l=0}^r \sum_{k=0}^l \left(H_{\sigma\nu}^l (\varepsilon_\tau)\circ J^r\iota\right) \frac{\pa f^\nu}{\pa q^j} \frac{\pa (x^\sigma_l \circ J^r\iota)}{\pa q^i_k} \eta^i_k \wedge \eta^j \wedge dt.
\end{equation}
\end{theorem}
\begin{proof}
By Theorem \ref{th6}, (\ref{comm}), $E_{J^r\iota{}^*\lambda} = {}^{Q}\!E(J^r\iota{}^*\lambda) = J^r\iota{}^* ({}^{\R^m}\!E(\lambda)) = J^r\iota{}^* E_\lambda,$
and analogously $H_{J^r\iota{}^*\varepsilon} = {}^{Q}\!E(J^r\iota{}^*\varepsilon) = J^r\iota{}^* ({}^{\R^m}\!E(\varepsilon)) = J^r\iota{}^* H_\varepsilon$. Applying now coordinate expressions for pull-backs with respect to $J^r\iota$ (Lemma \ref{lem:prolong} and \ref{lem:contact2}), to the forms $J^r\iota{}^* E_\lambda$ and $J^r\iota{}^* H_\varepsilon$, we get the required expressions of $E_{J^r\iota{}^*\lambda}$ and $H_{J^r\iota{}^*\varepsilon}$.
\end{proof}

We call (\ref{constrainedEL}) the $Q$-\textit{induced Euler--Lagrange form}, associated with Lagrangian $\lambda$, and (\ref{constrainedH}) the $Q$-\textit{induced Helmholtz form}, associated with source form $\varepsilon$.

\section{Variational submanifolds of Euclidean spaces:\\Second-order equations}
\label{sec:4}
An arbitrary system of second-order ordinary differential equations for an unknown section $t\to (t, x^\nu (t))$ of $\R\times\R^m$,
\begin{equation}\label{source}
\varepsilon_\sigma (t,x^\nu,\dot{x}^\nu,\ddot{x}^\nu) = 0,
\end{equation}
where the coefficients $\varepsilon_\sigma$, $1\leq\sigma\leq m$, are differentiable functions, defines a \textit{source form} on $J^2(\R\times\R^m)$ as 
\begin{equation}\label{Sforma}
\varepsilon = \varepsilon_\sigma \omega^\sigma\wedge dt,
\end{equation}
where $\omega^\sigma = dx^\sigma - \dot{x}^\sigma dt$. Recall that $\varepsilon$ (or the system $\varepsilon_\sigma$) is said to be \textit{(locally) variational}, if there exists a \textit{Lagrange function} $\mathscr{L}=\mathscr{L}(t,q^\sigma,\dot{q}^\sigma)$ such that
\begin{equation}\label{EL}
\varepsilon_\sigma = \frac{\pa\mathscr{L}}{\pa x^\sigma} - \frac{d}{dt}\frac{\pa\mathscr{L}}{\pa \dot{x}^\sigma} = \frac{\pa\mathscr{L}}{\pa x^\sigma} - \frac{\pa^2\mathscr{L}}{\pa t \pa \dot{x}^\sigma} - \frac{\pa^2\mathscr{L}}{\pa x^\nu \pa \dot{x}^\sigma}\dot{x}^\nu
- \frac{\pa^2\mathscr{L}}{\pa \dot{x}^\nu \pa \dot{x}^\sigma}\ddot{x}^\nu.
\end{equation}
Necessary and sufficient conditions for $\varepsilon$ to be variational are the well-known \textit{Helmholtz conditions}, a system of partial differential equations for $\varepsilon_\sigma$,
\begin{equation}\label{2ndH}
\begin{aligned}
& \frac{\pa\varepsilon_\sigma}{\pa\ddot{x}^\nu} - \frac{\pa\varepsilon_\nu}{\pa\ddot{x}^\sigma} = 0,\quad
 \frac{\pa\varepsilon_\sigma}{\pa\dot{x}^\nu} + \frac{\pa\varepsilon_\nu}{\pa\dot{x}^\sigma} -\frac{d}{dt}\left( \frac{\pa\varepsilon_\sigma}{\pa\ddot{x}^\nu} + \frac{\pa\varepsilon_\nu}{\pa\ddot{x}^\sigma} \right) =0,\\
& \frac{\pa\varepsilon_\sigma}{\pa{x}^\nu} - \frac{\pa\varepsilon_\nu}{\pa{x}^\sigma} -\frac{1}{2}\frac{d}{dt}\left( \frac{\pa\varepsilon_\sigma}{\pa\dot{x}^\nu} - \frac{\pa\varepsilon_\nu}{\pa\dot{x}^\sigma} \right) =0.
\end{aligned}
\end{equation}

\begin{remark}
The left-hand sides of equations (\ref{2ndH}), called the \textit{Helmholtz expressions}, are the coefficients of the Helmholtz form in the variational sequence, see Sec. \ref{sec:3}, (\ref{source2}), (\ref{Helmholtz}).
\end{remark}

Let $Q$ be an embedded submanifold of $\R^m$, $\dim Q=n$, $n\leq m$,  $\iota:\R\times Q\to\R\times\R^m$ be the canonical embedding, locally expressed by $t\circ\iota = t$, $x^\sigma\circ\iota = f^\sigma(q^i)$, $1\leq i\leq n$, $1\leq\sigma\leq m$ (Sec. \ref{sec:2}, (\ref{iota})).
The second jet prolongation $J^2\iota$ (Sec. \ref{sec:2}, (\ref{prolong})) has a chart expression given by Lemma \ref{lem:prolong}. If $\varepsilon=\varepsilon_\sigma \omega^\sigma \wedge dt$ is a source form on $J^2(\R\times\R^m)$, we get a source form $J^2\iota{}^*\varepsilon$ on $J^2(\R\times Q)$,
\begin{equation}\label{ind}
J^2\iota{}^*\varepsilon = (\varepsilon_\sigma\circ J^2\iota) J^2\iota{}^*\omega^\sigma \wedge J^2\iota{}^* dt =\tilde\varepsilon_i \eta^i\wedge dt ,
\end{equation}
where $\eta^i = dq^i - \dot{q}^i dt$, and the coefficients $\tilde\varepsilon_i = \tilde\varepsilon_i (t, q^j, \dot{q}^j, \ddot{q}^j)$, $1\leq i,j\leq n$, are given by
\begin{equation}\label{vlnka}
\tilde\varepsilon_i = (\varepsilon_\sigma \circ J^2\iota)\frac{\pa f^\sigma}{\pa q^i}.
\end{equation}
 
Local variationality is characterized by means of the Helmholtz conditions:
the induced source form $J^2\iota{}^*\varepsilon$ is locally variational if and only if the system of partial differential equations (\ref{2ndH}) is satisfied for functions $\tilde\varepsilon_i=\tilde\varepsilon_i(t, q^j, \dot{q}^j, \ddot{q}^j)$,
\begin{equation}\label{2ndH2}
\begin{aligned}
& \frac{\pa\tilde\varepsilon_i}{\pa\ddot{q}^j} - \frac{\pa\tilde\varepsilon_j}{\pa\ddot{q}^i} = 0,\quad
\frac{\pa\tilde\varepsilon_i}{\pa\dot{q}^j} + \frac{\pa\tilde\varepsilon_j}{\pa\dot{q}^i} -\frac{d}{dt}\left( \frac{\pa\tilde\varepsilon_i}{\pa\ddot{q}^j} + \frac{\pa\tilde\varepsilon_j}{\pa\ddot{q}^i} \right) =0,\\
& \frac{\pa\tilde\varepsilon_i}{\pa{q}^j} - \frac{\pa\tilde\varepsilon_j}{\pa{q}^i} -\frac{1}{2}\frac{d}{dt}\left( \frac{\pa\tilde\varepsilon_i}{\pa\dot{q}^j} - \frac{\pa\tilde\varepsilon_j}{\pa\dot{q}^i} \right) =0.
\end{aligned}
\end{equation}

We have now the following two problems:  (i) Suppose a submanifold $Q$ of $\R^m$ is fixed. Characterize a class of source forms $\varepsilon$ on $J^2\R^{m+1}$ such that $J^2\iota{}^*\varepsilon$ on $J^2(\R\times Q)$ is locally variational. (ii) Secondly, embedded submanifolds $Q$ of $\R^m$ may be regarded as unknowns of (\ref{2ndH2}) for a given source form $\varepsilon$. In this case the Helmholtz conditions represent a system of partial differential equations for unknown functions $f^\sigma$. 

We say that a submanifold $Q$ of $\R^m$ is \textit{variational} for a source form $\varepsilon$, if the induced source form $J^2\iota{}^*\varepsilon$ is locally variational.

Necessary and sufficient conditions for variationality of the induced source form $J^2\iota{}^*\varepsilon$ give an implicit solution of both problems (i) and (ii). In the following theorem we consider a source form $\varepsilon$ (\ref{Sforma}) on $J^2\R^{m+1}$, defined by functions $\varepsilon_\sigma=\varepsilon_\sigma (t,x^\nu,\dot{x}^\nu,\ddot{x}^\nu)$, and an embedded submanifold $Q$ of $\R^m$. $(U,\varphi)$, $\varphi = (t,q^i)$, is a fibered chart on $\R\times Q$, and $(V,\psi)$, $\psi = (t,u^j,v^\kappa)$, is a chart on $\R\times\R^{m}$ adapted to $Q$.

\begin{theorem}[Helmholtz conditions for induced variationality I]
\label{thHind}
The following conditions are equivalent:
\item[(a)] The induced source form $J^2\iota{}^*\varepsilon$ is locally variational.
\item[(b)] If $\iota$ is expressed by the equations $t\circ\iota =t$, $x^\sigma\circ \iota = f^\sigma (q^i)$, $1\leq\sigma\leq m$, then the functions $\varepsilon_\sigma$ and $f^\sigma$ satisfy
\begin{equation*}
\begin{aligned}
&\left( \frac{\pa\varepsilon_\sigma}{\pa\ddot{x}^\nu} - \frac{\pa\varepsilon_\nu}{\pa\ddot{x}^\sigma} \right)_{\circ J^2\iota} \left( \frac{\pa f^\sigma}{\pa q^i} \frac{\pa f^\nu}{\pa q^j} - \frac{\pa f^\nu}{\pa q^i} \frac{\pa f^\sigma}{\pa q^j} \right) =0,\\
&\left( \frac{\pa \varepsilon_\sigma}{\pa \dot{x}^\nu} + \frac{\pa \varepsilon_\nu}{\pa \dot{x}^\sigma} - \frac{d}{dt}\left( \frac{\pa\varepsilon_\sigma}{\pa \ddot{x}^\nu} + \frac{\pa\varepsilon_\nu}{\pa \ddot{x}^\sigma} \right) \right)_{\circ J^2\iota}
\left( \frac{\pa f^\sigma}{\pa q^i} \frac{\pa f^\nu}{\pa q^j} + \frac{\pa f^\nu}{\pa q^i} \frac{\pa f^\sigma}{\pa q^j} \right)\\
&- \left( \frac{\pa\varepsilon_\sigma}{\pa\ddot{x}^\nu} - \frac{\pa\varepsilon_\nu}{\pa\ddot{x}^\sigma} \right)_{\circ J^2\iota}
\left( \frac{\pa f^\nu}{\pa q^j}\frac{d}{dt}\frac{\pa f^\sigma}{\pa q^i}  + \frac{\pa f^\nu}{\pa q^i} \frac{d}{dt}\frac{\pa f^\sigma}{\pa q^j} \right.\\
&\left.
- \frac{\pa f^\sigma}{\pa q^j}\frac{d}{dt}\frac{\pa f^\nu}{\pa q^i}  - \frac{\pa f^\sigma}{\pa q^i} \frac{d}{dt}\frac{\pa f^\nu}{\pa q^j}
\right) =0,
\end{aligned}
\end{equation*}
and
\begin{equation*}
\begin{aligned}
&\left( \frac{\pa\varepsilon_\sigma}{\pa x^\nu} - \frac{\pa\varepsilon_\nu}{\pa x^\sigma} - \frac{1}{2}\frac{d}{dt}\left( \frac{\pa\varepsilon_\sigma}{\pa \dot{x}^\nu} - \frac{\pa\varepsilon_\nu}{\pa \dot{x}^\sigma} \right)\right)_{\circ J^2\iota}
\left( \frac{\pa f^\sigma}{\pa q^i}\frac{\pa f^\nu}{\pa q^j} - \frac{\pa f^\nu}{\pa q^i}\frac{\pa f^\sigma}{\pa q^j} \right)\\
&+ \frac{1}{2}\left( \frac{\pa\varepsilon_\sigma}{\pa \dot{x}^\nu} + \frac{\pa\varepsilon_\nu}{\pa \dot{x}^\sigma} - \frac{d}{dt}\left( \frac{\pa\varepsilon_\sigma}{\pa \ddot{x}^\nu} + \frac{\pa\varepsilon_\nu}{\pa \ddot{x}^\sigma} \right)\right)_{\circ J^2\iota}\\
&\cdot \left( \frac{\pa f^\sigma}{\pa q^i}\frac{d}{dt}\frac{\pa f^\nu}{\pa q^j} - \frac{\pa f^\sigma}{\pa q^j}\frac{d}{dt}\frac{\pa f^\nu}{\pa q^i}
+ \frac{\pa f^\nu}{\pa q^i}\frac{d}{dt}\frac{\pa f^\sigma}{\pa q^j} - \frac{\pa f^\nu}{\pa q^j}\frac{d}{dt}\frac{\pa f^\sigma}{\pa q^i}
\right)\\
&- \frac{1}{2} \frac{d}{dt}\left( \frac{\pa\varepsilon_\sigma}{\pa \ddot{x}^\nu} - \frac{\pa\varepsilon_\nu}{\pa \ddot{x}^\sigma} \right)_{\circ J^2\iota} \frac{d}{dt}\left( \frac{\pa f^\sigma}{\pa q^i}\frac{\pa f^\nu}{\pa q^j} - \frac{\pa f^\nu}{\pa q^i}\frac{\pa f^\sigma}{\pa q^j} \right)\\
&+ \left( \frac{\pa\varepsilon_\sigma}{\pa \ddot{x}^\nu} - \frac{\pa\varepsilon_\nu}{\pa \ddot{x}^\sigma} \right)_{\circ J^2\iota} 
\left( \frac{d}{dt}\left( \frac{\pa f^\sigma}{\pa q^j} \right) \frac{d}{dt}\left( \frac{\pa f^\nu}{\pa q^i} \right) \right.\\
&\left. - \frac{d}{dt}\left( \frac{\pa f^\nu}{\pa q^j} \right) \frac{d}{dt}\left( \frac{\pa f^\sigma}{\pa q^i} \right)\right)= 0.
\end{aligned}
\end{equation*}
\item[(c)] If $\iota$ is expressed by the equations $t\circ\iota =t$, $u^j\circ \iota = u^j|_{V\cap(\R\times Q)}$, $v^\kappa\circ \iota = c^\kappa$, where $1\leq j\leq n$, $n+1\leq \kappa\leq m$, and $c^\kappa$ are some real numbers, then the functions $\varepsilon_\sigma$ satisfy
\begin{equation}\label{Hsub}
\begin{aligned}
& \left( \frac{\pa\varepsilon_i}{\pa\ddot{u}^j} - \frac{\pa\varepsilon_j}{\pa\ddot{u}^i} \right)_{\circ J^2\iota} = 0,\quad
 \left( \frac{\pa\varepsilon_i}{\pa\dot{u}^j} + \frac{\pa\varepsilon_j}{\pa\dot{u}^i} -\frac{d}{dt}\left( \frac{\pa\varepsilon_i}{\pa\ddot{u}^j} + \frac{\pa\varepsilon_j}{\pa\ddot{u}^i} \right) \right)_{\circ J^2\iota}=0,\\
&\left( \frac{\pa\varepsilon_i}{\pa{u}^j} - \frac{\pa\varepsilon_j}{\pa{u}^i} -\frac{1}{2}\frac{d}{dt}\left( \frac{\pa\varepsilon_i}{\pa\dot{u}^j} - \frac{\pa\varepsilon_j}{\pa\dot{u}^i} \right)\right)_{\circ J^2\iota} =0.
\end{aligned}
\end{equation}
\end{theorem}
\begin{proof}
Equivalence of both conditions (b) and (c) with (a) follows from a straightforward substitution of the expressions for $\tilde\varepsilon_i$ (\ref{vlnka}) into (\ref{2ndH2}) with respect to the considered charts.
\end{proof}
\begin{remark}
Note that the Helmholtz expressions (\ref{2ndH}) and (\ref{2ndH2}) appear in Theorem \ref{thHind}: condition (b) contains linear combinations of expressions (\ref{2ndH}), and in (c), (\ref{Hsub}) contains expressions (\ref{2ndH2}), restricted to a submanifold. In particular, the equivalence conditions of Theorem \ref{thHind} show that local variationality of the induced system $\tilde\varepsilon_i$ (\ref{vlnka}) determines the structure of the initial system on $J^2\R^{m+1}$, defined by functions $\varepsilon_\sigma=\varepsilon_\sigma (t, x^\nu,\dot{x}^\nu,\ddot{x}^\nu)$, and \textit{vice-versa}. Namely, the system $\varepsilon_\sigma$, $\sigma=1,2,\dots, m$, must contain a \textit{subsystem} of $n$ equations for $n$ dependent variables, which is locally variational when contracted to a submanifold. This observation gives a straightforward method to construct systems of differential equations on Euclidean spaces, which are \textit{not} variational but induce locally variational systems on submanifolds.
\end{remark}
\begin{corollary}\label{cor:Lagrangian}
Suppose a source form $\varepsilon=\varepsilon_\sigma \omega^\sigma\wedge dt$ {\rm (\ref{Sforma})} on $J^2\R^{m+1}$ is variational. Then the induced source form 
$J^2\iota{}^*\varepsilon =\tilde\varepsilon_i \eta^i\wedge dt$ {\rm (\ref{ind})} on $J^2(\R\times Q)$ is globally variational. Moreover, if $\mathscr{L}=\mathscr{L} (t,x^\nu,\dot{x}^\nu)$ is a Lagrange function of the system $\varepsilon_\sigma$, then $\mathscr{L}\circ J^1\iota$ is a Lagrange function of the induced system $\tilde\varepsilon_i$.
\end{corollary}
\begin{proof}
From Theorem \ref{thHind} it follows that $\tilde\varepsilon_i$ satisfies the Helmholtz conditions (\ref{2ndH2}) provided $\varepsilon_\sigma$ is satisfies (\ref{2ndH}). Suppose now that for some function $\mathscr{L}:J^1\R^{m+1}\to\R$, $\varepsilon_\sigma$ are expressed as the Euler--Lagrange expressions (\ref{EL}). We show that the Euler--Lagrange expressions of the function $\mathscr{L}\circ J^1\iota:J^1(\R\times Q)\to\R$ coincide with $\tilde\varepsilon_i$. Using the chart expression of $J^1\iota$ we get
\begin{equation*}
\begin{aligned}
&\frac{\pa (\mathscr{L}\circ J^1\iota)}{\pa q^i} - \frac{d}{dt} \frac{\pa (\mathscr{L}\circ J^1\iota)}{\pa \dot{q}^i}\\
&=\left(\frac{\pa \mathscr{L}}{\pa x^\sigma}\right)_{\circ J^1\iota} \frac{\pa f^\sigma}{\pa q^i}
+ \left(\frac{\pa \mathscr{L}}{\pa \dot{x}^\sigma}\right)_{\circ J^1\iota}  \frac{\pa^2 f^\sigma}{\pa q^i \pa q^j} \dot{q}^j 
-\frac{d}{dt}\left( \frac{\pa\mathscr{L}}{\pa\dot{x}^\sigma}_{\circ J^1\iota} \frac{\pa f^\sigma}{\pa q^i} \right)\\
&= \left( \frac{\pa\mathscr{L}}{\pa x^\sigma} - \frac{d}{dt}\frac{\pa\mathscr{L}}{\pa\dot{x}^\sigma} \right)_{\circ J^2\iota} \frac{\pa f^\sigma}{\pa q^i} = (\varepsilon_\sigma\circ J^2\iota)\frac{\pa f^\sigma}{\pa q^i} = \tilde\varepsilon_i,
\end{aligned}
\end{equation*}
as required.
\end{proof}

It is easily seen that variationality of a source form $\varepsilon$ (\ref{Sforma}) implies that its coefficients $\varepsilon_\sigma$ are linear in the second derivatives, i.e. 
\begin{equation}\label{Ex}
\varepsilon_\sigma = A_\sigma + B_{\sigma\nu} \ddot{x}^\nu
\end{equation}
for some functions $A_{\sigma}$, $B_{\sigma\nu}$ depending $t,\ x^\sigma$ and $\dot{x}^\sigma$ only. In this case the Helmholtz conditions (\ref{2ndH}) are equivalent with the following conditions for $A_\sigma$ and $B_{\sigma\nu}$ (cf. Sarlet \cite{Sarlet}),
\begin{equation}\label{Hspecial}
\begin{aligned}
& B_{\sigma\nu}=B_{\nu\sigma},\quad \frac{\pa B_{\sigma\kappa}}{\pa\dot{x}^\nu} = \frac{\pa B_{\nu\kappa}}{\pa\dot{x}^\sigma},\\
&\frac{\pa A_{\sigma}}{\pa \dot{x}^\nu} + \frac{\pa A_{\nu}}{\pa \dot{x}^\sigma}
-2\left( \frac{\pa B_{\sigma\nu}}{\pa t} +  \frac{\pa B_{\sigma\nu}}{\pa x^\kappa} \dot{x}^\kappa \right) = 0,\\
&\frac{\pa A_{\sigma}}{\pa {x}^\nu} - \frac{\pa A_{\nu}}{\pa {x}^\sigma}
-\frac{1}{2}\left( \frac{\pa}{\pa t} \left( \frac{\pa A_{\sigma}}{\pa \dot{x}^\nu} - \frac{\pa A_{\nu}}{\pa \dot{x}^\sigma} \right) + \frac{\pa}{\pa x^\kappa}\left( \frac{\pa A_{\sigma}}{\pa \dot{x}^\nu} - \frac{\pa A_{\nu}}{\pa \dot{x}^\sigma} \right) \dot{x}^\kappa \right) =0.
\end{aligned}
\end{equation}
%%%%%%%%%%%%%%%%%

From now on we shall consider second-order systems (\ref{Ex}) with  coefficients $A_\sigma$ and  $B_{\sigma\nu}$ of the variables $x^\sigma$ and $\dot{x}^\sigma$, and \textit{not} depending on time variable $t$ explicitly. It is straightforward to check that the coefficients of the induced source form $J^2\iota{}^*\varepsilon =\tilde\varepsilon_j \eta^j\wedge dt$ (\ref{ind}) are expressed by $\tilde\varepsilon_j = P_j + Q_{jk}\ddot{q}^k$, where
\begin{equation}\label{PQ}
\begin{aligned}
P_j = \tilde A_\sigma \frac{\pa f^\sigma}{\pa q^j} + \tilde B_{\sigma\nu} \frac{\pa f^\sigma}{\pa q^j}\frac{\pa^2 f^\nu}{\pa q^i \pa q^k} \dot{q}^i \dot{q}^k,\quad
Q_{jk} = \tilde B_{\sigma\nu} \frac{\pa f^\sigma}{\pa q^j} \frac{\pa f^\nu}{\pa q^k},
\end{aligned}
\end{equation}
and $\tilde A_\sigma = A_\sigma\circ J^1\iota$, $\tilde B_{\sigma\nu} = B_{\sigma\nu}\circ J^1\iota$. Suppose $\varepsilon$ is a source form on $J^2\R^{m+1}$, defined by functions {\rm (\ref{Ex})}, $(U,\varphi)$, $\varphi = (t,q^i)$, is a fibered chart on $\R\times Q$, and $(V,\psi)$, $\psi = (t,u^j,v^\kappa)$, is a chart on $\R\times\R^{m}$ adapted to $Q$. 
%%%%

\begin{theorem}[Helmholtz conditions for induced variationality II]
\label{thind}
The following conditions are equivalent:
\item[(a)] The induced source form $J^2\iota{}^*\varepsilon$ is locally variational.
\item[(b)] If $\iota$ is expressed by the equations $t\circ\iota =t$, $x^\sigma\circ \iota = f^\sigma (q^i)$, $1\leq\sigma\leq m$, then the functions $\tilde A_\sigma$, $\tilde B_{\sigma\nu}$, and $f^\sigma$ satisfy
\begin{equation*}
\begin{aligned}
& (\tilde B_{\sigma\nu} - \tilde B_{\nu\sigma} ) \frac{\pa f^\sigma}{\pa q^j} \frac{\pa f^\nu}{\pa q^k}  =0,\quad
 \left(
\frac{\partial \tilde B_{\sigma\nu}}{\partial \dot{q}^i} \frac{\partial f^\sigma}{\partial q^j}
-
\frac{\partial \tilde B_{\sigma\nu}}{\partial \dot{q}^j} \frac{\partial f^\sigma}{\partial q^i}
\right) \frac{\partial f^\nu}{\partial q^k} 
=0,
\\
&
\frac{\partial \tilde A_\sigma}{\pa \dot{q}^j} \frac{\partial f^\sigma}{\partial q^i}
+
\frac{\partial \tilde A_\sigma}{\pa \dot{q}^i} \frac{\partial f^\sigma}{\partial q^j}
-2
\frac{\pa \tilde B_{\sigma\nu}}{\pa q^k}
\frac{\partial f^\sigma}{\partial q^i}
\frac{\partial f^\nu}{\partial q^j}
\dot{q}^k
\\
&\ +
\left(
\frac{\partial \tilde B_{\sigma\nu}}{\pa \dot{q}^j} \frac{\partial f^\sigma}{\partial q^i}
+ 
\frac{\partial \tilde B_{\sigma\nu}}{\pa \dot{q}^i} \frac{\partial f^\sigma}{\partial q^j}
\right)
\frac{\partial^2 f^\nu}{\partial q^k\partial q^l} \dot{q}^k\dot{q}^l 
= 0,
\end{aligned}
\end{equation*}
%%%%%%%%%%%%%%
and
\begin{equation*}
\begin{aligned}
&
\frac{\pa \tilde A_\sigma}{\pa {q}^j} \frac{\partial f^\sigma}{\partial q^i}
-  \frac{\pa \tilde A_\sigma}{\pa {q}^i} \frac{\partial f^\sigma}{\partial q^j}
-
\frac{1}{2}\frac{\partial}{\partial q^k}
\left(
\frac{\pa \tilde A_\sigma}{\pa \dot{q}^j} \frac{\partial f^\sigma}{\partial q^i}
-  \frac{\pa \tilde A_\sigma}{\pa \dot{q}^i} \frac{\partial f^\sigma}{\partial q^j}
\right) \dot{q}^k
\\
&
- (\tilde B_{\sigma\nu} - \tilde B_{\nu\sigma}) \frac{\pa^2 f^\sigma}{\pa q^i\pa q^l} \frac{\pa^2 f^\nu}{\pa q^j\pa q^k} \dot{q}^k \dot{q}^l\\
&- \frac{\pa \tilde B_{\sigma\nu}}{\pa q^k} \left( \frac{\partial f^\sigma}{\partial q^i} \frac{\partial^2 f^\nu}{\partial q^l \pa q^j} -  \frac{\partial f^\sigma}{\partial q^j} \frac{\partial^2 f^\nu}{\partial q^l \pa q^i}\right) \dot{q}^k \dot{q}^l \\
&+
\left( \frac{\pa \tilde B_{\sigma\nu}}{\pa q^j} \frac{\partial f^\sigma}{\partial q^i} - \frac{\pa \tilde B_{\sigma\nu}}{\pa q^i} \frac{\partial f^\sigma}{\partial q^j} \right) \frac{\partial^2 f^\nu}{\partial q^k \pa q^l} \dot{q}^k \dot{q}^l \\
&-
\frac{1}{2}\frac{\pa}{\pa q^k}
\left( \left(
\frac{\pa \tilde B_{\sigma\nu}}{\pa \dot{q}^j} \frac{\partial f^\sigma}{\partial q^i} - \frac{\pa \tilde B_{\sigma\nu}}{\pa \dot{q}^i} \frac{\partial f^\sigma}{\partial q^j} \right) \frac{\partial^2 f^\nu}{\partial q^s \pa q^l} \dot{q}^s \dot{q}^l \right) \dot{q}^k =0.
\end{aligned}
\end{equation*}
\item[(c)] If $\iota$ is expressed by the equations $t\circ\iota =t$, $u^j\circ \iota = u^j|_{V\cap(\R\times Q)}$, $v^\kappa\circ \iota = c^\kappa$, where $1\leq j\leq n$, $n+1\leq \kappa\leq m$, and $c^\kappa$ are some real numbers, then $\tilde A_j =A_j \circ J^1\iota$, $\tilde B_{jk}= B_{jk}\circ J^1\iota$, $1\leq j,k\leq n$, satisfy
\begin{equation}\label{HspecialT}
\begin{aligned}
& \tilde B_{jk}=\tilde B_{kj},\quad \frac{\pa \tilde B_{jl}}{\pa\dot{q}^k} = \frac{\pa \tilde B_{kl}}{\pa\dot{q}^j},\quad \frac{\pa \tilde A_{j}}{\pa \dot{q}^k} + \frac{\pa \tilde A_{k}}{\pa \dot{q}^j}
-2 \frac{\pa \tilde B_{jk}}{\pa q^l} \dot{q}^l = 0,\\
&\frac{\pa \tilde A_{j}}{\pa {q}^k} - \frac{\pa \tilde A_{k}}{\pa {q}^j}
-\frac{1}{2} \frac{\pa}{\pa q^l}\left( \frac{\pa \tilde A_{j}}{\pa \dot{q}^k} - \frac{\pa \tilde A_{k}}{\pa \dot{q}^j} \right) \dot{q}^l =0.
\end{aligned}
\end{equation}
\end{theorem}
\begin{proof}
The proof follows from a direct application of the Helmholtz variationality conditions (\ref{Hspecial}), applied to coefficients (\ref{PQ}) of $J^2\iota{}^*\varepsilon$.
\end{proof}
%%%%%%%%%%%%%%%
\begin{corollary}
\label{th1dim}
Let $Q$ be a one-dimensional embedded submanifold of $\R^m$. Then the induced source form $J^2\iota{}^*\varepsilon$ on an open subset of $J^2(\R\times Q)$ is locally variational if and only if the functions $A_\sigma$, $B_{\sigma\nu}$, and $f^\sigma$ satisfy
\begin{equation}\label{2-Helm}
\begin{aligned}
&\left(\frac{\pa A_\sigma}{\pa \dot{x}^\nu}\right)_{\circ J^1\iota} \frac{\pa f^\nu}{\pa q} \frac{\pa f^\sigma}{\pa q}
+ \left(\frac{\pa B_{\sigma\kappa}}{\pa \dot{x}^\nu} - \frac{\pa B_{\sigma\nu}}{\pa \dot{x}^\kappa} \right)_{\circ J^1\iota} 
 \dot{q} \frac{\pa f^\nu}{\pa q} \frac{\pa f^\sigma}{\pa q}
 \frac{d}{dt} \frac{\pa f^\kappa}{\pa q} \\
 &+ \left(B_{\sigma\nu} - B_{\nu\sigma} \right)_{\circ J^1\iota} 
\frac{\pa f^\sigma}{\pa q} \frac{d}{dt} \frac{\pa f^\nu}{\pa q}
 - \left( \frac{\pa B_{\sigma\nu}}{\pa x^\kappa} \right)_{\circ J^1\iota}
 \frac{d f^\kappa}{dt}  \frac{\pa f^\nu}{\pa q} \frac{\pa f^\sigma}{\pa q} = 0.
\end{aligned}
\end{equation}
\end{corollary}
\begin{remark}\label{rem6}
We point out that the equivalent conditions of Theorem \ref{thind} restrict not only the structure of the initial source form $\varepsilon$, but also a submanifold $Q$ of $\R^m$ to be found. In particular, we observe that (\ref{2-Helm}) is satisfied identically for systems $\varepsilon_\sigma$ such that (i) $A_\sigma=A_\sigma(x^\kappa)$, (ii) $B_{\sigma\nu}=B_{\sigma\nu}(\dot{x}^\kappa)$, and
\begin{equation}\label{nutna}
B_{\sigma\nu} = B_{\nu\sigma},\quad \frac{\pa B_{\sigma\nu}}{\pa \dot{x}^\kappa} = \frac{\pa B_{\sigma\kappa}}{\pa \dot{x}^\nu}.
\end{equation}
In other words, second-order systems of ordinary equations linear in second derivatives which satisfy conditions (i), (ii), and (\ref{nutna}) are variational when contracted to arbitrary \textit{one-dimensional} submanifold of the Euclidean space $\R^m$. Conditions (\ref{nutna}) belong to the set of the Helmholtz conditions (\ref{Hspecial}) for $\varepsilon_\sigma$, nevertheless (\ref{nutna}) together with (i) and (ii) do \textit{not} imply (local) variationality of the system $\varepsilon_\sigma$ (\ref{Ex}).
\end{remark}
\begin{remark}[Forces]
Important particular case of (\ref{Ex}) are systems solved with respect to the second derivatives,
\begin{equation}\label{solved}
\ddot{x}^\sigma = F^\sigma,
\end{equation}
where $1\leq \sigma,\nu \leq m$, and $F^\sigma= F^\sigma (x^\nu,\dot{x}^\nu)$ are given functions (called \textit{mechanical forces}). To rewrite  (\ref{solved}) into the form $ \varepsilon_\sigma =0$, we put $B_{\sigma\nu} = \delta_{\sigma\nu}$ (the Kronecker symbol), and $A_\sigma = -\delta_{\sigma\nu} F^\nu$. From Corollary \ref{th1dim} we obtain a necessary and sufficient condition for a one-dimensional submanifold $Q$ of $\R^m$ to be a variational submanifold for (\ref{solved}),
\begin{equation*}
\sum_{\sigma} \left( \frac{\pa F^\sigma}{\pa \dot{x}^\nu} \right)_{\circ J^1\iota} 
\frac{\pa f^\sigma}{\pa q} \frac{\pa f^\nu}{\pa q}  = 0.
\end{equation*}
In particular, if the forces $F^\sigma$ do \textit{not} depend on velocities $\dot{x}^\nu$, then clearly any one-dimensional submanifold is variational for (\ref{solved}).
\end{remark}
%%%%%%%%%%%%%%%%%
\section{Variational submanifolds: Examples}
\label{sec:5}
We illustrate the general theory on simple examples of \textit{holonomic} constraints in mechanics. Consider the Euclidean space $\R\times\R^3$ endowed with the canonical coordinates $(t,x,y,z)$, and the associated fibered coordinates on the jet prolongations $J^1(\R\times\R^3)$,  $J^2(\R\times\R^3)$, denoted by $(t,x,y,z,\dot x,\dot y,\dot z)$ and $(t,x,y,z,\dot x,\dot y,\dot z,\ddot x,\ddot y,\ddot z)$, respectively.

In Examples 1 and 2 we consider two topologically \textit{non}-equivalent, embedded submanifolds of $\R^3$, the sphere $S^2$ and the M\"obius strip $M_{r,a}$. With the help of the variational sequence theory (Sec. \ref{sec:3}) the concepts of local and global variationality of the constrained systems of equations are studied. Then in Examples 3--8 we study \textit{one}- and \textit{two}-dimensional \textit{variational} submanifolds of Euclidean spaces for concrete systems of second-order differential equations.
\begin{example}[Induced variational equations on $2$-sphere]
\label{exam1}
Denote by $S^2_{r_0}$ (resp. $S^2$) the two-dimensional sphere in $\R^3$ of radius $r_0>0$ (resp. unit sphere) with centre in the origin. In the standard sense, the open submanifold $\R\times(\R^3 \setminus \{ (0,0,0)\})$ of $\R^4$ is endowed with a smooth atlas, which consists of two \textit{spherical} charts adapted to the embedded submanifold $\R\times S^2_{r_0}$ for every $r_0$. The equations
$t=t,\ x= r\cos\varphi\sin\vartheta,\ y= r\sin\varphi\sin\vartheta,\ z= r\cos\vartheta,
$
define a diffeomorphism $U\ni (t, r, \varphi,\vartheta)\to (t, x(r, \varphi,\vartheta),y(r, \varphi,\vartheta),z(r, \varphi,\vartheta))\in V$, where for instance $U=\{(t, r, \varphi,\vartheta)\in\R^3\,|\, t\in\R,\ r>0,\  0<\varphi < 2\pi,\ 0<\vartheta < \pi \}$ maps onto an open subset $V=\R\times\R^3\setminus \{(x,y,z)\in\R^3\, |\, x\geq 0, y=0 \}$ of $\R^4$. The inverse transformation $V\ni (t, x,y,z)\to (t, r, \varphi, \vartheta)\in U$, is a diffeomorphism expressed as
\begin{equation*}
\begin{aligned}
& t=t,\quad r=\sqrt{x^2+y^2+z^2},\quad  \vartheta=\arccos{\frac{z}{\sqrt{x^2+y^2+z^2}}} \\
& \varphi=
\begin{cases}
\arccos\frac{x}{\sqrt{x^2+y^2}},
 & y>0,
 \\
2\pi-\arccos\frac{x}{\sqrt{x^2+y^2}},
 & y<0,
\\
 \pi,
 & y=0,\ x<0.
\end{cases}
\end{aligned}
\end{equation*}
$(V,\Psi)$, $\Psi = (t, r, \varphi,\vartheta)$, is a chart on $\R\times (\R^3\setminus\{ (0,0,0)\}$, \textit{adapted} to $\R\times S^2_{r_0}$. We complete $(V,\Psi)$ by a rotated chart $(\bar V,\bar\Psi)$, where $\bar V = \R\times (\R^3\setminus\{ (x,y,z)\in\R^3\,|\, x\leq 0,\ z=0 \})$, $\bar\Psi = (\bar t, \bar r, \bar\varphi,\bar\vartheta)$, in order to obtain an atlas on $\R\times (\R^3\setminus\{ (0,0,0)\}$. The coordinate transformation
$\Psi\circ \bar{\Psi}^{-1}: \bar{\Psi} (V\cap\bar{V}) \to \Psi (V\cap\bar{V})$, is expressed by
\begin{equation*}
t=\bar{t},\quad r=\bar{r},\quad \cos\varphi = \frac{-\cos\bar{\varphi}\sin\bar\vartheta}{\sqrt{1-\sin^2\bar\varphi\sin^2\bar\vartheta}},\quad
\cos\vartheta = -\sin\bar\varphi\sin\bar\vartheta.
\end{equation*}

Consider the canonical embedding
$
\iota : \R\times S^2 \to \R\times\R^3
$
over the identity of the real line $\R$. Let $\varepsilon$ be a source form on $J^2(\R\times\R^3)$, 
$
\varepsilon = \varepsilon_x \omega^x\wedge dt + \varepsilon_y \omega^y\wedge dt + \varepsilon_z \omega^z\wedge dt,
$
where $\omega^x = dx - \dot{x}dt$, $\omega^y = dy - \dot{y}dt$, $\omega^z = dz - \dot{z}dt$. We discuss variationality of $J^2\iota{}^*\varepsilon$.

(i) If $\varepsilon$ is locally variational, then by Corollary \ref{corglobal} it is also globally variational since $H^2_{\mathrm{deR}}(\R\times\R^3)$ is trivial. Thus, $\varepsilon$ coincides with the Euler--Lagrange form $E_\lambda$, associated with a Lagrangian $\lambda = \mathscr{L} dt$, where $\mathscr{L}:J^1(\R\times\R^3)\to \R$ is a globally defined function. By Corollary \ref{cor:Lagrangian}, the induced source form $J^2\iota{}^*\varepsilon$ is \textit{globally} variational and coincides with the Euler--Lagrange form, associated with the Lagrangian $J^2\iota{}^*\lambda$. From Theorem \ref{th7}, (a), we get an expression of the constrained Euler--Lagrange form with respect to $(V^1, \Psi^1)$,
\begin{equation*}
\begin{aligned}
& J^2\iota{}^*\varepsilon =J^2\iota{}^* E_\lambda\\
&=\left( -(E_x(\mathscr{L})\circ J^2\iota)\sin\varphi\sin\vartheta + (E_y(\mathscr{L})\circ J^2\iota)\cos\varphi\sin\vartheta \right) \eta^\varphi \wedge dt\\
&+ 
\left( (E_x(\mathscr{L})\circ J^2\iota)\cos\varphi\cos\vartheta + (E_y(\mathscr{L})\circ J^2\iota)\sin\varphi\cos\vartheta \right.\\
&\left. - (E_z(\mathscr{L})\circ J^2\iota)\sin\vartheta \right) \eta^\vartheta \wedge dt,
\end{aligned}
\end{equation*}
where $\eta^\varphi = d\varphi -\dot\varphi dt$, $\eta^\vartheta = d\vartheta -\dot\vartheta dt$, and
\begin{equation*}
\begin{aligned}
E_x(\mathscr{L}) &= \frac{\pa \mathscr{L}}{\pa x} - \frac{d}{dt}\frac{\pa\mathscr{L}}{\pa\dot{x}},\quad
E_y(\mathscr{L}) = \frac{\pa \mathscr{L}}{\pa y} - \frac{d}{dt}\frac{\pa\mathscr{L}}{\pa\dot{y}},\\
E_z(\mathscr{L}) &= \frac{\pa \mathscr{L}}{\pa z} - \frac{d}{dt}\frac{\pa\mathscr{L}}{\pa\dot{z}}.
\end{aligned}
\end{equation*}

(ii) Suppose $\varepsilon$ is \textit{not} variational. If $\varepsilon= \varepsilon_\varphi \omega^\varphi \wedge dt + \varepsilon_\vartheta\omega^\vartheta \wedge dt + \varepsilon_r\omega^r \wedge dt$ in $(V^1, \Psi^1)$, we get necessary and sufficient conditions (Theorem \ref{thHind}) for local variationality of $J^2\iota{}^*\varepsilon$ as
\begin{equation}\label{spherevar}
\begin{aligned}
&\left(\frac{\pa\varepsilon_\varphi}{\pa\ddot{\vartheta}} - \frac{\pa\varepsilon_\vartheta}{\pa\ddot{\varphi}}\right)_{\circ J^2\iota} = 0,\\
&\left(\frac{\pa\varepsilon_\varphi}{\pa\dot{\varphi}} - \frac{d}{dt} \frac{\pa\varepsilon_\varphi}{\pa\ddot{\varphi}}\right)_{\circ J^2\iota} =  0,\quad 
\left(\frac{\pa\varepsilon_\vartheta}{\pa\dot{\vartheta}} - \frac{d}{dt} \frac{\pa\varepsilon_\vartheta}{\pa\ddot{\vartheta}}\right)_{\circ J^2\iota} =  0,\\
& \left(\frac{\pa\varepsilon_\varphi}{\pa\dot{\vartheta}} + \frac{\pa\varepsilon_\vartheta}{\pa\dot{\varphi}} 
 - \frac{d}{dt} \left( \frac{\pa\varepsilon_\varphi}{\pa\ddot{\vartheta}}+\frac{\pa\varepsilon_\vartheta}{\pa\ddot{\varphi}}\right)\right)_{\circ J^2\iota} =  0,\\
& \left(\frac{\pa\varepsilon_\varphi}{\pa{\vartheta}} - \frac{\pa\varepsilon_\vartheta}{\pa{\varphi}} 
 - \frac{1}{2}\frac{d}{dt} \left( \frac{\pa\varepsilon_\varphi}{\pa\dot{\vartheta}}-\frac{\pa\varepsilon_\vartheta}{\pa\dot{\varphi}}\right)\right)_{\circ J^2\iota} =  0.
\end{aligned}
\end{equation}
Since the second de Rham cohomology group of $S^2$ is \textit{non}-trivial,  $H^2_{\mathrm{deR}}(\R\times S^2) = \R$, we observe by Corollary \ref{corglobal} that the local variationality conditions (\ref{spherevar}) do \textit{not} assure global variationality of $J^2\iota{}^*\varepsilon$. Apparently, one can easily find other examples of embedded submanifolds in Euclidean spaces possessing topological obstructions for global variationality.
\end{example}
\begin{example}[Induced variational equations on M\"obius strip]
\label{exam2}
Let $M_{r,a}$ denotes the M\"obius strip in $\R^3$ (\textit{without} boundary) of radius $r$ and wideness $2a$, where $0<a<r$. We introduce a smooth atlas on open subset $\R\times (\R^3\setminus \{ (0,0,z)\,|\, z\in\R \})$ of $\R^4$, consisting of adapted charts to embedded submanifold $\R\times M_{r,a}$ as follows. The equations $t=t$, and
\begin{equation*}
\begin{aligned}
x&=r\cos\varphi + \tau\cos\frac{\varphi}{2}\cos\varphi -\kappa\sin\frac{\varphi}{2}\cos\varphi,\\
y&=r\sin\varphi + \tau\cos\frac{\varphi}{2}\sin\varphi -\kappa\sin\frac{\varphi}{2}\sin\varphi,\\
z&= \tau\sin\frac{\varphi}{2} + \kappa\cos\frac{\varphi}{2},\quad
0<\varphi < 2\pi,\ \  -\infty < \tau < +\infty,
\\
\end{aligned}
\end{equation*}
where 
\begin{equation}\label{plocha}
-\infty < \kappa < \frac{\cos(\varphi/2)}{\sin(\varphi/2)}\tau +\frac{r}{\sin(\varphi/2)},
\end{equation}
define a diffeomorphism
\begin{equation*}
U\ni (t,\varphi,\tau,\kappa)\to (t, x(\varphi,\tau,\kappa), y(\varphi,\tau,\kappa), z(\varphi,\tau,\kappa))\in V,
\end{equation*}
where $V=\R\times (\R^3\setminus \{ [0,+\infty)\times\{0\}\times\R \})$, $U=\R\times U_0$, and $U_0$ is an open subset of $\R^3$, given by condition (\ref{plocha}). We get a chart $(V,\Psi)$, $\Psi=(t,\varphi,\tau,\kappa)$, on $\R\times (\R^3\setminus \{ (0,0,z)\,|\, z\in\R \})$, adapted to submanifold $\R\times M_{r,a}$, which is characterized by $\kappa =0$. Another adapted chart to $\R\times M_{r,a}$, is given by $(\bar V,\bar\Psi)$, $\bar\Psi=(t,\bar\varphi,\bar\tau,\bar\kappa)$, where $\bar V = \R\times ( \R^3\setminus \{ (-\infty, 0]\times \{ 0\}\times \R \})$, and
\begin{equation*}
\begin{aligned}
x&=r\cos\bar\varphi + \bar\tau\cos\frac{\bar\varphi}{2}\cos\bar\varphi -\bar\kappa\sin\frac{\bar\varphi}{2}\cos\bar\varphi,\\
y&=r\sin\bar\varphi + \bar\tau\cos\frac{\bar\varphi}{2}\sin\bar\varphi -\bar\kappa\sin\frac{\bar\varphi}{2}\sin\bar\varphi,\\
z&= \bar\tau\sin\frac{\bar\varphi}{2} + \bar\kappa\cos\frac{\bar\varphi}{2},
\quad -\pi < \bar\varphi <\pi,\   -\infty <\kappa < +\infty,\\
\end{aligned}
\end{equation*}
where
\begin{equation*}
\frac{\sin(\varphi/2)}{\cos(\varphi/2)}\kappa - \frac{r}{\cos(\varphi/2)} < \tau < +\infty.
\end{equation*}
Then $(V,\Psi)$ and $(\bar V,\bar\Psi)$ define an adapted atlas to $\R\times M_{r,a}$, obeying chart transformation 
\begin{align*}
\bar\Psi\circ\Psi^{-1} (t,\varphi,\tau,\kappa) = & \begin{cases}
(t,\varphi,\tau,\kappa), & \varphi\in (0,\pi), \\
(t,\varphi-2\pi,-\tau,-\kappa), & \varphi\in (\pi,2\pi).
\end{cases}
\end{align*}

Consider the canonical embedding $\iota: \R\times M_{r,a} \to \R\times \R^3.$ Let $\varepsilon$ be a source form on $J^2(\R\times\R^3)$, expressed by
\begin{equation*}
\varepsilon = \varepsilon_x \omega^x\wedge dt + \varepsilon_y \omega^y\wedge dt + \varepsilon_z \omega^z\wedge dt.
\end{equation*}
With respect to the associated chart $(V_1,\Psi_1)$, $\Psi_1 = (t,\varphi,\tau)$, on $\R\times M_{r,a}$, we get $J^2\iota{}^*\varepsilon = \tilde\varepsilon_\varphi \eta^\varphi\wedge dt + \tilde\varepsilon_\tau \eta^\tau\wedge dt$, where $\eta^\varphi = d\varphi - \dot{\varphi} dt$, $\eta^\tau = d\tau - \dot{\tau} dt$, and
\begin{equation*}
\begin{aligned}
\tilde\varepsilon_\varphi &= (\varepsilon_x\circ J^2\iota) (-r\sin\varphi + (\tau/2)\sin(\varphi/2) -3\tau \cos^2(\varphi/2)\sin(\varphi/2))\\
&+ (\varepsilon_y\circ J^2\iota) (r\cos\varphi + \tau\cos(\varphi/2) -3\tau \cos(\varphi/2)\sin^2(\varphi/2)) \\
&+ (\varepsilon_z\circ J^2\iota)  (\tau/2)\cos(\varphi/2),\\
\tilde\varepsilon_\tau &=(\varepsilon_x\circ J^2\iota) ( \cos(\varphi/2) -2 \cos(\varphi/2)\sin^2(\varphi/2))\\
&+2 (\varepsilon_y\circ J^2\iota) \sin(\varphi/2)\cos^2(\varphi/2)
+ (\varepsilon_z\circ J^2\iota) \sin(\varphi/2).
\end{aligned}
\end{equation*}
Necessary and sufficient conditions (the Helmholtz conditions) for local variationality of $J^2\iota{}^*\varepsilon$ now read (cf. Theorem \ref{thHind}),
\begin{equation}\label{mobvar}
\begin{aligned}
&\frac{\pa\tilde\varepsilon_\varphi}{\pa\ddot{\tau}} - \frac{\pa\tilde\varepsilon_\tau}{\pa\ddot{\varphi}} = 0,\quad
\frac{\pa\tilde\varepsilon_\varphi}{\pa\dot{\varphi}} - \frac{d}{dt} \frac{\pa\tilde\varepsilon_\varphi}{\pa\ddot{\varphi}} =  0,\quad 
\frac{\pa\tilde\varepsilon_\tau}{\pa\dot{\tau}} - \frac{d}{dt} \frac{\pa\tilde\varepsilon_\tau}{\pa\ddot{\tau}} =  0\\
& \frac{\pa\tilde\varepsilon_\varphi}{\pa\dot{\tau}} + \frac{\pa\tilde\varepsilon_\tau}{\pa\dot{\varphi}} 
 - \frac{d}{dt} \left( \frac{\pa\tilde\varepsilon_\varphi}{\pa\ddot{\tau}}+\frac{\pa\tilde\varepsilon_\tau}{\pa\ddot{\varphi}}\right) =  0,\\
& \frac{\pa\tilde\varepsilon_\varphi}{\pa{\tau}} - \frac{\pa\tilde\varepsilon_\tau}{\pa{\varphi}} 
 - \frac{1}{2}\frac{d}{dt} \left( \frac{\pa\tilde\varepsilon_\varphi}{\pa\dot{\tau}}-\frac{\pa\tilde\varepsilon_\tau}{\pa\dot{\varphi}}\right) =  0.
\end{aligned}
\end{equation}
Moreover, the second de Rham cohomology group of $M_{r,a}$ is trivial, hence $H^2_{\mathrm{deR}}(\R\times M_{r,a})$ is also trivial. By Corrolary \ref{corglobal} it follows that $J^2\iota{}^*\varepsilon$ is automatically \textit{globally variational} provided (\ref{mobvar}) are satisfied.
\end{example}
%%%%%%
\begin{example}[One-dimensional submanifolds of $\R^m$ are variational]
\label{exam3}
In Remark \ref{rem6} we described a class of second-order systems of differential equations which possess \textit{any} one-dimensional submanifold of $\R^m$ as a variational submanifold. Another class of equations of this kind is constructed as follows.

Let $Q$ be an arbitrary embedded one-dimensional submanifold of $\R^m$. We search for a source form $\varepsilon$ on $J^2(\R\times\R^m)$ such that $J^2\iota{}^*\varepsilon$ is (locally) variational. Let $(V,\psi)$, $\psi=(u, v^\alpha)$, $2\leq \alpha \leq m$, be a chart on $\R^m$ adapted to $Q$, and the canonical embedding $\iota:\R\times Q \to \R\times \R^m$ over $\id_\R$, has the equations $t\circ\iota = t$, $u\circ\iota = u|_{V\cap Q}$, $v^\alpha\circ\iota = 0$. We put $\varepsilon=\varepsilon_0 \eta\wedge dt +\varepsilon_\alpha \omega^\alpha \wedge dt$, where $\eta = du - \dot{u}dt$, $\omega^\alpha = dv^\alpha -\dot{v}^\alpha dt $, $2\leq\alpha \leq m$, for arbitrary functions $\varepsilon_\alpha$ on $J^2(\R\times\R^m)$ and for $\varepsilon_0 = A + \dot{u}^2 \ddot{u}$, where $A:V^1\to\R$ is such that $A\circ J^1\iota$ depends on coordinate $u$ only. Clearly, $\varepsilon$ may be chosen \textit{non}-variational. The induced source form has an expression 
$
J^2\iota{}^*\varepsilon = \left( A\circ J^2\iota + \dot{u}^2\ddot{u} \right) \eta|_{V\cap Q}\wedge dt 
$,
and one can easily see that $J^2\iota{}^*\varepsilon$ is locally variational since the Helmholtz condition (\ref{2-Helm}) is satisfied identically.
\end{example}
%%%%%%
\begin{example}[Variational submanifolds: Circles $S^1_r$ in $\R^2$]
\label{exam4}
We give a second-order system, which is \textit{not} variational but every circle $S^1_r$ in $\R^2$ is its variational submanifold. Let $\varepsilon$ be a source form on $J^2(\R\times\R^2)$, $\varepsilon = \varepsilon_\sigma \omega^\sigma \wedge dt$, defined by functions 
\begin{equation}\label{priklad1}
\varepsilon_1 = \dot{x}^2 + \dot{x}^1\left(\dot{x}^1\ddot{x}^1 + \dot{x}^2\ddot{x}^2 \right),\quad
\varepsilon_2 = - \dot{x}^1 + \dot{x}^2 \left( \dot{x}^1\ddot{x}^1 + \dot{x}^2\ddot{x}^2 \right).
\end{equation}
In the expression (\ref{Ex}) we have $A_1=\dot{x}^2$, $A_2=- \dot{x}^1$, and $B_{\sigma\nu} = \dot{x}^\sigma \dot{x}^\nu $, $\sigma,\nu = 1,2$. It is easy to check that (\ref{priklad1}) is \textit{non}-variational since the second  condition of (\ref{2ndH}) is not satisfied. Indeed,
\begin{equation*}
\frac{\pa\varepsilon_\sigma}{\pa\dot{x}^\sigma} -\frac{d}{dt}\frac{\pa\varepsilon_\sigma}{\pa\ddot{x}^\sigma} =
 \sum_{\nu} \dot{x}^\nu\ddot{x}^\nu + \dot{x}^\sigma \ddot{x}^\sigma
 - \frac{d}{dt} \left( \dot{x}^\sigma \right)^2
 = \sum_{\nu\neq\sigma} \dot{x}^\nu\ddot{x}^\nu;
\end{equation*}
remark also that the second condition of (\ref{nutna}) is also not satisfied. 
Consider the \textit{polar} coordinates $(r,\varphi)$ on $(0,+\infty)\times (0,2\pi)$, adapted to circles $S^1_r$ in $\R^2$. The mapping $U \ni (t,r,\varphi)\to (t,x^1,x^2)\in V$, given by $t=t$, $x^1=r\cos\varphi$, $x^2=r\sin\varphi$,
where 
\begin{equation*}
\begin{aligned}
U= \R\times (0,+\infty)\times (0,2\pi),\quad
V= \R\times\left(\R^2\setminus\{ (x,y)\,|\, x\geq 0,y=0 \}\right),
\end{aligned}
\end{equation*}
defines a coordinate transformation between the Cartesian and the polar charts, with respect to which the embedding $\iota:\R\times S^1_{r} \to \R\times \R^2$ has an expression $t\circ\iota =t$, $x^1\circ\iota = f^1(\varphi)= r\cos\varphi$, $x^2\circ\iota =f^2(\varphi)=r \sin\varphi$. One can easily verify that $J^2\iota{}^*\varepsilon$ on $J^2(\R\times S^1_{r_0})$ is locally variational. Indeed, we get $J^2\iota{}^*\varepsilon = \tilde\varepsilon \eta\wedge dt$, where $\eta = d\varphi -\dot\varphi dt$, and $\tilde\varepsilon = r^4 \dot{\varphi}^2\ddot{\varphi}$.
The Helmholtz expression for $\tilde\varepsilon$ reads,
\begin{equation*}
\frac{\pa\tilde\varepsilon}{\pa \dot{\varphi}} 
- \frac{d}{dt} \frac{\pa\tilde\varepsilon}{\pa \ddot{\varphi}} 
= 2r^4\dot{\varphi}\ddot{\varphi} - r^4\frac{d}{dt}\left( \dot{\varphi}^2 \right) =0.
\end{equation*} 
The system defined by functions (\ref{priklad1}) belongs to class of mechanical systems with \textit{generalised potential force} (cf. Dreizler and L\"{u}dde \cite{Dreizler}).
\end{example}
%%%%%
\begin{example}[Dumped oscillator has no variational submanifold]
\label{exam5}
Consider the equations of motion of the \textit{two-dimensional free oscillator} in $\R\times\R^2$ (cf. Landau and Lifshitz \cite{Landau}),
\begin{equation}\label{2osc}
\begin{aligned}
 m_{11}\ddot{x} + m_{12}\ddot{y} + k_{11}x + k_{12}y = 0,\quad m_{12}\ddot{x} + m_{22}\ddot{y} + k_{12}x + k_{22}y = 0,
\end{aligned}
\end{equation}
where $k_{11},k_{12},k_{22}$ are coefficients of the potential energy, and all the mass coefficients $m_{11},m_{12},m_{22}$ are equal 1. System (\ref{2osc}) is variational and coincides with the Euler--Lagrange equations of Lagrange function $$\mathscr{L} = ( m_{11}\dot{x}^2 + 2 m_{12}\dot{x}\dot{y} + m_{22}\dot{y}^2 - k_{11}x^2 -2 k_{12}xy - k_{22}y^2)/2$$ (kinetic \textit{minus} potential energy). Adding to the left-hand sides of (\ref{2osc}) the derivatives $\pa F /\pa \dot{x}$ and $\pa F /\pa \dot{y}$ of the \textit{dissipative force} $F = ( \alpha_{11}\dot{x}^2 + 2\alpha_{12} \dot{x}\dot{y} + \alpha_{22} \dot{y}^2)/2$, which is a positive definite quadratic form, we get the damped oscillator equations,
\begin{equation}\label{dampedEq}
\varepsilon_1(x,y,\dot{x},\dot{y},\ddot{x},\ddot{y}) = 0,\quad \varepsilon_2(x,y,\dot{x},\dot{y},\ddot{x},\ddot{y}) = 0,
\end{equation}
where $\varepsilon_1 =\ddot{x} +\ddot{y}  + k_{11}x + k_{12}y + \alpha_{11}\dot{x} + \alpha_{12}\dot{y} $, and $\varepsilon_2 =\ddot{x} + \ddot{y} + k_{12}x + k_{22}y + \alpha_{12}\dot{x} + \alpha_{22}\dot{y}$. System (\ref{dampedEq}) is \textit{not} variational. Indeed, the Helmholtz conditions (\ref{2ndH}) for $\varepsilon_1,\varepsilon_2$ imply that $\alpha_{11}=\alpha_{12}=\alpha_{22}=0$, which contradicts our assumption on dissipative coefficients $\alpha_{ij}$ of the positive definite form $F$.

Consider the canonical embedding $\iota:\R\times Q \to \R\times\R^2$, where $Q$ is a \textit{one}-dimensional submanifold of $\R^2$. Denote by $u,v$ the \textit{adapted coordinates} to $Q$, defined on an open subset of $\R^2$, such that $u\circ J^2\iota=u$, $v\circ J^2\iota=0$. The pull-back $J^2\iota{}^*\varepsilon$ of $\varepsilon$, where $\varepsilon = \varepsilon_1 \omega^1\wedge dt + \varepsilon_2\omega^2\wedge dt$, $\omega^1=dx - \dot{x}dt$, $\omega^2=dy - \dot{y}dt$, has a chart expression $J^2\iota{}^*\varepsilon = \tilde\varepsilon \eta\wedge dt$, where $\eta = du - \dot{u} dt$, and
\begin{equation*}
\begin{aligned}
\tilde\varepsilon &=  \left.\left( \left(  \frac{\pa x}{\pa u} \right)^2 +  2\frac{\pa x}{\pa u} \frac{\pa y}{\pa u} + \left(  \frac{\pa y}{\pa u} \right)^2 \right)\right|_{v=1} \ddot{u}\\
  &+  \left. \left(  \frac{\pa^2 x}{\pa u^2}\frac{\pa x}{\pa u} + \frac{\pa^2 y}{\pa u^2} \frac{\pa x}{\pa u} + \frac{\pa^2 x}{\pa u^2}\frac{\pa y}{\pa u} + \frac{\pa^2 y}{\pa u^2}\frac{\pa y}{\pa u}  \right) \right|_{v=1}  \dot{u}^2   \\
  &+\left. \left(  \alpha_{11}\left( \frac{\pa x}{\pa u} \right)^2 + 2\alpha_{12}\frac{\pa x}{\pa u} \frac{\pa y}{\pa u} + \alpha_{22}\left( \frac{\pa y}{\pa u}\right)^2 \right)\right|_{v=1} \dot{u}\\
  &+\left( k_{11} x (u,v) + k_{12}y(u,v) \right) \left. \frac{\pa x}{\pa u} \right|_{v=1}  
  + \left( k_{12}x (u,v) + k_{22}y(u,v)  \right) \left. \frac{\pa y}{\pa u} \right|_{v=1}.
\end{aligned}
\end{equation*}
$J^2\iota{}^*\varepsilon$ is locally variational if and only if
\begin{equation}\label{quadrat}
\begin{aligned}
\frac{\pa\tilde\varepsilon}{\pa \dot{u}} 
&- \frac{d}{dt} \frac{\pa\tilde\varepsilon}{\pa \ddot{u}}\\
&= \alpha_{11}\left.\left( \frac{\pa x}{\pa u} \right)^2\right|_{v=1} + 2\alpha_{12}\left.\frac{\pa x}{\pa u}\right|_{v=1} \left.\frac{\pa y}{\pa u}\right|_{v=1} + \alpha_{22}\left.\left( \frac{\pa y}{\pa u}\right)^2\right|_{v=1}
\end{aligned}
\end{equation}
vanishes (cf. (\ref{2ndH2}), Theorem \ref{thHind}). This is, however, not possible as (\ref{quadrat}) is a positive definite quadratic form which does \textit{not} vanish for any $Q$. Analogously, we observe that there is \textit{no} \textit{two}-dimensional variational submanifold of $\R^3$ for the three-dimensional damped oscillator.
\end{example}
%%%%%%%%
\begin{example}[Two-dimensional submanifolds of $\R^m$ are variational]
\label{exam6}
We generalize Example \ref{exam3} to \textit{two}-dimensional submanifolds of $\R^m$, $m\geq 3$. Let $Q$ be an embedded two-dimensional submanifold of $\R^m$, and $(V,\psi)$, $\psi=(u,v,w^\alpha)$, $3\leq\alpha\leq m$, be a chart on $\R^m$ adapted to $Q$. The canonical embedding $\iota:\R\times Q\to \R\times\R^m$ has the equations $t\circ \iota=t$, $u\circ\iota=u|_{V\cap Q}$, $v\circ\iota=v|_{V\cap Q}$, $w^\alpha\circ\iota=0$, $3\leq\alpha\leq m$. Let  $\varepsilon= \varepsilon_\sigma\omega^\sigma\wedge dt$ on $J^2(\R\times\R^m)$ is defined by
\begin{equation}\label{exdim2nonvar}
\varepsilon_\sigma=B_{\sigma1}\ddot u+B_{\sigma2}\ddot v+B_{\sigma\alpha}\ddot{w}^\alpha,
\end{equation}
where
\begin{equation*}
\begin{aligned}
&B_{11}=\dot{v}^2/2,
&&B_{12}=B_{21}=\dot{u}\dot{v},
&&B_{1\alpha}=B_{\alpha 1}=\dot{u}\dot{w}^\alpha,
\\
&B_{22}=\dot{u}^2/2,
&&B_{2\alpha}=B_{\alpha 2}=\dot{v}\dot{w}^\alpha,
&&B_{\beta\alpha}=B_{\alpha\beta}=\dot{w}^\alpha\dot{w}^\beta/2,
\end{aligned}
\end{equation*}
$1\leq \sigma\leq m$, $3\leq \alpha,\beta\leq m$.
It is easy to check that (\ref{exdim2nonvar}) is \textit{not} variational (cf. Theorem \ref{thind}).
However, the induced source $J^2\iota^\ast \varepsilon = \tilde\varepsilon_1\eta^1 \wedge dt+\tilde\varepsilon_2\eta^2\wedge dt$, where
\begin{equation*}
\tilde\varepsilon_1
=\frac{1}{2}(\dot v^2 \ddot u + 2\dot u\dot v\ddot v),\quad
\tilde\varepsilon_2= \frac{1}{2}(2\dot u \dot v \ddot u + \dot u^2\ddot v),
\end{equation*}
is locally variational. Moreover, $J^2\iota^\ast \varepsilon$ is the Euler--Lagrange form, associated with the Lagrangian
\begin{equation*}
\lambda=\mathscr{L}(\dot{u},\dot{v}) dt=-\frac{1}{4}\dot{u}^2\dot{v}^2dt.
\end{equation*}
\end{example}
%%%%%%%
\begin{example}[A variational submanifold: the sphere $S^2$ in $\R^3$]
\label{exam7}
Let us consider a source form $\varepsilon=\varepsilon_x\omega^x\wedge dt + \varepsilon_y\omega^y\wedge dt + \varepsilon_z\omega^z\wedge dt$ on $J^2(\R\times\R^3)$, where $\omega^x=dx-\dot{x}dt$, $\omega^y=dy-\dot{y}dt$, $\omega^z=dz-\dot{z}dt$, such that
\begin{equation}\label{spherecan}
\varepsilon_x=\sqrt{x^2+y^2+z^2}+\ddot{x},\quad \varepsilon_y=\sqrt{x^2+y^2+z^2}+\ddot{y},\quad
\varepsilon_z=\sqrt{x^2+y^2+z^2}+\ddot{z}.
\end{equation}
System (\ref{spherecan}) is \textit{not} variational, and we claim that it has a variational submanifold the unit sphere $S^2$ in $\R^3$. Consider the canonical embedding $\iota:\R\times S^2\to\R\times\R^3$, expressed by $t\circ\iota = t$, $x\circ\iota=\cos \varphi \sin \vartheta$, $y\circ\iota=\sin \varphi \sin \vartheta$, $z\circ\iota=\cos \vartheta$,
with respect to the canonical coordinates $(t,x,y,z)$ on $\R\times\R^3$, and the associated chart $(U,\Phi)$, $\Phi=(t,\varphi,\vartheta)$, on $\R\times S^2$. We get $J^2\iota^\ast\varepsilon=\tilde\varepsilon_\varphi\eta^\varphi\wedge dt+\tilde\varepsilon_\vartheta\eta^\vartheta\wedge dt$,
where $\eta^\varphi=d\varphi-\dot{\varphi}dt$, $\eta^\vartheta=d\vartheta-\dot{\vartheta}dt$, $\tilde\varepsilon_1=P_1+\sin^2 \vartheta\ddot{\varphi}$, $\tilde\varepsilon_2=P_2+\ddot{\vartheta}$, and
\begin{equation*}
\begin{aligned}
&P_1=\sin \vartheta(\cos \varphi -\sin \varphi +2\cos \vartheta \dot{\varphi}\dot{\vartheta}),\\
&P_2=\sin \varphi \cos \vartheta+\cos \varphi \cos \vartheta-\sin \vartheta-\sin \vartheta \cos \vartheta \dot{\varphi}^2.
\end{aligned}
\end{equation*}
A direct calculation shows that $J^2\iota^\ast\varepsilon$ is locally variational. Moreover, $J^2\iota^\ast \varepsilon$ is the Euler--Lagrange form, associated with the Lagrangian $\tilde{\mathscr{L}} dt$, where
\begin{equation*}
\begin{aligned}
\tilde{\mathscr{L}}= -\frac{1}{2}(\sin^2\vartheta \,\dot{\varphi}^2+\dot{\vartheta}^2)+\cos \varphi\sin \vartheta+\sin \varphi\sin \vartheta+\cos \vartheta.
\end{aligned}
\end{equation*}
Note that the Lagrange function $\tilde{\mathscr{L}}$ comes from the mechanical Lagrangian $\mathscr{L} dt$, defined on $J^1(\R\times\R^3)$, where $\mathscr{L}=-(\dot{x}^2+\dot{y}^2+\dot{z}^2)/2+x+y+z$ satisfies $\tilde {\mathscr{L}}=\mathscr{L}\circ J^1\iota$.
\end{example}

%%%%%%%%%%%%%%%%
\begin{example}[Gyroscopic type system]
\label{exam8}
The equations of motion of the \textit{3-dimensional gyroscopic type} system in the Euclidean space $\R\times\R^3$ read (cf. Landau and Lifshitz \cite{Landau}),
\begin{equation}\label{3gyr}
 \ddot{x} = \alpha \dot{y} + \beta \dot{z},\quad 
 \ddot{y} = - \alpha \dot{x} + \gamma \dot{z},\quad
 \ddot{z} = - \beta \dot{x} - \gamma \dot{y},
\end{equation}
where $\alpha$, $\beta$, and $\gamma$ are arbitrary functions of the positions $x,y,z$. Note the right-hand side of (\ref{3gyr}) belongs to a class of mechanical forces, expressible as a vector product of velocity and position vectors; basic examples are the \textit{Lorentz force}, or the \textit{Coriolis force} (cf. Dreizler and L\"{u}dde \cite{Dreizler}).  Computing the Helmholtz expressions (\ref{2ndH}), we observe that the system,
$ \varepsilon_1 = \ddot{x} - \alpha \dot{y} - \beta \dot{z},\ 
  \varepsilon_2 = \ddot{y} + \alpha \dot{x} - \gamma \dot{z},\ 
   \varepsilon_3 = \ddot{z} + \beta \dot{x} + \gamma \dot{y},$
on $J^2(\R\times\R^3)$ is locally variational if and only if
\begin{equation*}
\frac{\pa\alpha}{\pa z}  -\frac{\pa\beta}{\pa y} + \frac{\pa\gamma}{\pa x} = 0.
\end{equation*}

Let $Q$ be an embedded two-dimensional submanifold of $\R^3$, $\iota:\R\times Q\to\R\times\R^3$ the canonical embedding. Denote by $u,v$, and $w$ the \textit{adapted coordinates} to $Q$, defined on an open subset of $\R^3$, such that $\iota$ has the equations $u\circ\iota = u|_{Q}$, $v\circ\iota = v|_{Q}$, and $w\circ\iota = 0$. Then $J^2\iota{}^*\varepsilon=\tilde\varepsilon_1\eta^1\wedge dt
+\tilde\varepsilon_2\eta^2\wedge dt$, where $\eta^1=du-\dot{u}dt$, $\eta^2=dv-\dot{v}dt$, and 
\begin{equation*}
\tilde \varepsilon_1=P_1+Q_{11}\ddot{u}+Q_{12}\ddot{v},\quad \tilde \varepsilon_2=P_2+Q_{12}\ddot{u}+Q_{22}\ddot{v},
\end{equation*}
where
\begin{equation*}
\begin{aligned}
&
P_1=
\left( {\frac {\partial f^x}{\partial u}}  \right)  
\left( {\frac {\partial^2 f^x}{\partial {u}^
2}}  \right) \dot{u}^2
+2
 \left( {\frac {\partial f^x}{\partial u}}  \right)  
\left( {\frac {\partial^2 f^x}{\partial v \partial u}}  \right) 
\dot{u}\dot{v}
+ \left( {\frac {\partial f^x}{\partial u}} \right)  
\left( {\frac {\partial^2 f^x}{\partial {v}^
2}}  \right) \dot{v}^2
\\
&
+ \left( {\frac {\partial f^y}{\partial u}}  \right)  
\left( {\frac {\partial^2 f^y}{\partial u^2}}  \right) \dot{u}^{2}
+2 \left( {\frac {\partial f^y}{\partial u}}  \right)  
\left( {\frac {
\partial ^2 f^y}{\partial v\partial u}} \right) \dot{u}\dot{v}
+ \left( {\frac {\partial f^y}{\partial u}} \right)  
\left( {\frac {
\partial ^{2} f^y}{\partial v^{2}}}  \right) \dot{v}^2
\\&
+ \left( {\frac {\partial f^z}{\partial u}}  \right)  
\left( {\frac {\partial^{2} f^z}{\partial u^2}}  \right) \dot{u}^{2}
+2 \left( {\frac {\partial f^z}{\partial u}} \right)  
\left( {\frac {\partial^2 f^z}{\partial v\partial u}}
 \right) \dot{u}\dot{v}
+
 \left( {\frac {\partial f^z}{\partial u}}  \right)  
\left( {\frac {\partial^2 f^z}{\partial v^2}}  \right) \dot{v}^2
\\
&
+\left(\frac{\partial f^y}{\partial u}\frac{\partial f^x}{\partial v}
-\frac{\partial f^x}{\partial u}\frac{\partial f^y}{\partial v}
\right)\alpha\dot{v}
+
\left(\frac{\partial f^z}{\partial u}\frac{\partial f^x}{\partial v}
-\frac{\partial f^x}{\partial u}\frac{\partial f^z}{\partial v}
\right)\beta\dot{v}
\\
&
+
\left(\frac{\partial f^z}{\partial u}\frac{\partial f^y}{\partial v}
-\frac{\partial f^y}{\partial u}\frac{\partial f^z}{\partial v}
\right)\gamma\dot{v},
\end{aligned}
\end{equation*}
\begin{equation*}
\begin{aligned}
&P_2=
\left( {\frac {\partial f^x}{\partial v}}  \right)  
\left( {\frac {\partial^2 f^x}{\partial {u}^
2}}  \right) \dot{u}^2
+2
 \left( {\frac {\partial f^x}{\partial v}}  \right)  
\left( {\frac {\partial^2 f^x}{\partial v \partial u}}  \right) 
\dot{u}\dot{v}
+ \left( {\frac {\partial f^x}{\partial v}} \right)  
\left( {\frac {\partial^2 f^x}{\partial {v}^
2}}  \right) \dot{v}^2
\\
&
+ \left( {\frac {\partial f^y}{\partial v}}  \right)  
\left( {\frac {\partial^2 f^y}{\partial u^2}}  \right) \dot{u}^{2}
+
2 \left( {\frac {\partial f^y}{\partial v}}  \right)  
\left( {\frac {\partial ^2 f^y}{\partial v\partial u}} \right) \dot{u}\dot{v}
+ 
\left( {\frac {\partial f^y}{\partial v}} \right)  
\left( {\frac {
\partial ^{2} f^y}{\partial v^2}}  \right) \dot{v}^2
\\&
+ \left( {\frac {\partial f^z}{\partial v}}  \right)  
\left( {\frac {\partial^{2} f^z}{\partial u^2}}  \right) \dot{u}^{2}
+
2 \left( {\frac {\partial f^z}{\partial v}} \right)  
\left( {\frac {\partial^2 f^z}{\partial v\partial u}}
 \right) \dot{u}\dot{v}
+
 \left( {\frac {\partial f^z}{\partial v}}  \right)  
\left( {\frac {\partial^2 f^z}{\partial v^2}}  \right) \dot{v}^2
\\
&
+\left(\frac{\partial f^y}{\partial u}\frac{\partial f^x}{\partial v}
-\frac{\partial f^x}{\partial u}\frac{\partial f^y}{\partial v}
\right)\alpha\dot{v}
+
\left(\frac{\partial f^z}{\partial u}\frac{\partial f^x}{\partial v}
-\frac{\partial f^x}{\partial u}\frac{\partial f^z}{\partial v}
\right)\beta\dot{v}
\\
&
+
\left(\frac{\partial f^z}{\partial u}\frac{\partial f^y}{\partial v}
-\frac{\partial f^y}{\partial u}\frac{\partial f^z}{\partial v}
\right)\gamma\dot{v},
\end{aligned}
\end{equation*}
and
\begin{equation*}
\begin{aligned}
%%%%%%%%%%%%%%%%%%%%%%%%%%%%%
&Q_{11}=\left(\frac{\partial f^x}{\partial u}\right)^2
+\left(\frac{\partial f^y}{\partial u}\right)^2
+\left(\frac{\partial f^z}{\partial u}\right)^2,
\quad
Q_{22}=\left(\frac{\partial f^x}{\partial v}\right)^2
+\left(\frac{\partial f^y}{\partial v}\right)^2
+\left(\frac{\partial f^z}{\partial v}\right)^2,
\\
&
Q_{12}=
\frac{\partial f^x}{\partial u}\frac{\partial f^x}{\partial v}
+\frac{\partial f^y}{\partial u}\frac{\partial f^y}{\partial v}
+\frac{\partial f^z}{\partial u}\frac{\partial f^z}{\partial v}.
\end{aligned}
\end{equation*}
The source form $J^2\iota{}^*\varepsilon$ satisfies the Helmholtz conditions (Theorem \ref{thind}) for arbitrary functions $\alpha$, $\beta$, $\gamma$ depending on the positions $x,y,z$. Hence any two-dimensional submanifold of $\R^3$ is a variational submanifold for the gyroscopic type system (\ref{3gyr}).
%%%%%%%
\end{example}


\begin{thebibliography}{}
%
\bibitem{And}
I.M. Anderson and T. Duchamp, On the existence of global variational principles, Am. J. Math. 102 (1980), 781--867.
\bibitem{Brajer}
J. Brajer\v{c}\'ik and D. Krupka, Variational principles for locally variational forms, J. Math. Phys. 46 (2005), 052903. 
\bibitem{Dreizler}
R.M. Dreizler and C.S. L\"{u}dde, \textit{Theoretical Mechanics}, Theoretical Physics Vol. 1, Graduate Texts in Physics, Springer-Verlag, Berlin, 2011.
\bibitem{Praha}
D. Krupka, On the local structure of the Euler--Lagrange mapping of the calculus of variations, in: Proc. Conf. Diff. Geom. Appl., Charles University, Prague, pp. 181--188, 1981; arXiv:math-ph/0203034.
\bibitem{Seq}
D. Krupka, Variational sequences on finite order jet spaces, in: Diff. Geom. Appl., Prof. Conf., 27 August - 2 September, 1989, Czechoslovakia, World Scientific, Singapore, pp. 236--254, 1990.
\bibitem{SeqMech}
D. Krupka, Variational sequences in mechanics, Calc. Var. 5 (1997), 557--583.
\bibitem{Book}
D. Krupka, \textit{Introduction to Global Variational Geometry}, Atlantis Press, Amsterdam, 2015.
\bibitem{Handbook}
D. Krupka and D. Saunders (Eds.), \textit{Handbook of Global Analysis}, Elsevier, Amsterdam, 2008.
\bibitem{Landau}
L.D. Landau and E.M. Lifshitz, \textit{Mechanics}, 2nd Edition, Course of Theoretical Physics, Vol. 1, Pergamon Press, Oxford, 1969.
\bibitem{Sarlet}
W. Sarlet, The Helmholtz conditions revisited. A new approach to the inverse problem of Lagrangian dynamics, J. Phys. A: Math. Gen. 15 (1982), 1503--1517.
\bibitem{Takens}
F. Takens, A global version of the inverse problem of the calculus of variations, J. Diff. Geom. 14 (1979), 543--562.
\bibitem{Acta}
Z. Urban and D. Krupka, Variational sequences in mechanics on Grassmann fibrations, Acta Appl. Math. 112, No. 2 (2010), 225--249.
\bibitem{Urban}
Z. Urban, Variational Principles for Immersed Submanifolds, in: D. Zenkov (Ed.), \textit{The Inverse Problem of the Calculus of Variations}, Atlantis Press, Amsterdam, pp. 103--170, 2015.
\bibitem{Voicu}
N. Voicu and D. Krupka, Canonical variational completion of differential equations, J. Math. Phys. 56 (2015), 043507.
\bibitem{VU}
J. Voln\'a and Z. Urban, The interior Euler--Lagrange operator in field theory, Math. Slovaca 65,  No. 6 (2015), 1427--1444.
\bibitem{Warner} 
F.W. Warner, \textit{Foundations of Differentiable Manifolds and Lie Groups}, 2nd Ed., Springer, New York, 1983.
\end{thebibliography}
\end{document}